\documentclass[11pt]{amsart}

\usepackage[margin=2.5cm]{geometry}
\usepackage{amsmath,amssymb,amsfonts,amsthm,amsfonts,mathrsfs,amsrefs}
\usepackage{epic}
\usepackage{amsopn,amscd, graphicx}
\usepackage{color, transparent} 
\usepackage[usenames, dvipsnames]{xcolor}	
\usepackage[colorlinks, breaklinks,
bookmarks = false,
linkcolor = Blue,
urlcolor = Green,
citecolor = Green,
hyperfootnotes = false
]{hyperref}
\usepackage{palatino, mathpazo}
\usepackage{multirow}
\usepackage[all,cmtip]{xy}
\usepackage{stmaryrd}
\usepackage{tikz-cd}
\usepackage{yhmath}
\usepackage{enumerate}

 1
 1
 1

\newtheorem{theorem}{Theorem}[section]

\newtheorem{proposition}{Proposition}[section]
\newtheorem{lemma}{Lemma}[section]
\theoremstyle{definition}
\newtheorem{definition}{Definition}[section]

\newcommand{\C}{\mathscr{C}}

\newcommand{\ddbar}{\overline\partial}
\newcommand{\pr}{\partial}
\newcommand{\ol}{\overline}
\newcommand{\Td}{\widetilde}

\newcommand{\To}{\rightarrow}

\numberwithin{equation}{section}
\title{On the second coefficient of the asymptotic expansion of Boutet de Monvel--Sj\"ostrand}
\date{}
\author{Chin-Yu Hsiao}

\address{Institute of Mathematics, Academia Sinica and 
National Center for Theoretical Sciences,\newline
    \mbox{\quad}\,Astronomy-Mathematics Building, 
No. 1, Sec. 4, Roosevelt Road, Taipei 10617, Taiwan}
\email{chsiao@math.sinica.edu.tw or chinyu.hsiao@gmail.com}

\author[Wei-Chuan Shen]{Wei-Chuan Shen}
\address{Universit{\"a}t zu K{\"o}ln,  Mathematisches Institut,
    Weyertal 86-90,   50931 K{\"o}ln, Germany}
\email{wshen@math.uni-koeln.de}
\thanks{The first author was partially supported by Taiwan Ministry of Science and Technology projects 108-2115-M-001-012-MY5 and 109-2923-M-001-010-MY4 and Academia Sinica Career Development Award. The second author was supported by the SFB/TRR 191 "Symplectic Structures in Geometry, Algebra and Dynamics", funded by the DFG (Projektnummer 281071066 – TRR 191).}
\begin{document}
\begin{abstract}
In this paper, we calculate the second coefficient of the asymptotic expansion of Boutet de Monvel--Sj\"ostrand. 
\end{abstract}
\maketitle
\tableofcontents
\section{Introduction and the main result}

Let $(X,T^{1,0}X)$ be a CR manifold and let
$\overline{\partial}_b$ be the tangential Cauchy--Riemann operator on $X$. The orthogonal projection
$\Pi:L^2(X)\To{\rm Ker\,}\overline{\partial}_b$ is called the Szeg\H{o} projection, and we call
its distribution kernel $\Pi(x,y)$ the Szeg\H{o} kernel. The study of the Szeg\H{o}
kernel is a classical subject in several complex variables and CR geometry.
We recall the following classical result of Boutet de Monvel and Sj\"ostrand~\cite{BouSj76} about the description of the Szeg\H{o} kernel:

\begin{theorem}[Boutet de Monvel–Sj\"ostrand]\label{Boutet-Sjostrand-Hsiao formula}
Let $X$ be a compact orientable strongly pseudovoncex CR manifold of real dimension $2n+1$, $n\geq1$. 
Suppose that  $\overline{\partial}_b: {\rm Dom\,}\ddbar_b\subset L^2(X)\To L^2_{(0,1)}(X)$ has closed range. Let $D\subset X$ 
be an open coordinate patch with local coordinates $x=(x_1,\ldots,x_{2n+1})$. Then, 
\begin{equation}
    \label{Boutet-Sjostrand formula}
    \Pi(x,y)\equiv\int_0^\infty e^{it\phi(x,y)}a(x,y,t)dt\mod \mathscr{C}^\infty(D\times D),
\end{equation}
where 
\begin{equation}\label{e-gue201217yyd}
\begin{split}
&\phi(x,y)\in \C^\infty(D\times D),\ \  {\rm Im\,}\phi\geq0,\\
&\mbox{$\phi(x,y)=0$ if and only if $x=y$},\\
&\mbox{$d_x\phi(x,x)=-d_y\phi(x,x)=-\omega_0(x)$, for every $x\in D$},
\end{split}
\end{equation}
\begin{equation}\label{e-gue201217yydI}
\begin{split}
&\mbox{$a(x,y,t)\sim\sum^{+\infty}_{j=0}a_j(x,y)t^{n-j}$ in $S^n_{1,0}(D\times D\times\mathbb R_+)$},\\
&\mbox{$a_j(x,y)\in\C^\infty(D\times D)$, $j=0,1,\ldots$},\\
&\mbox{$a_0(x,x)\neq0$, for every $x\in D$},
\end{split}
\end{equation}
where $\omega_0\in\C^\infty(X,T^*X)$ is the global one form given by \eqref{e-gue201221yyd}.
\end{theorem}

We refer the reader to Sections~\ref{s-ssna},~\ref{s-gue201221yyd} for the setup, notations and terminology used in Theorem~\ref{Boutet-Sjostrand-Hsiao formula}. 
The Boutet de Monvel and Sj\"ostrand's description of the Szeg\H{o} kernel has profound impact in several complex variables, CR and complex geometry and geometric quantization theory, e.t.c. For example,  Catlin~\cite{Catlin1999BKE} and Zelditch~\cite{Zelditch1998BKE} established Bergman kernel asymptotic expansions for high power of positive holomorphic line bundles by using Theorem~\ref{Boutet-Sjostrand-Hsiao formula} (see also~\cite{HM14}). It should be mentioned that Tian \cite{Tian} obtained the leading term of Bergman kernel asymptotics for high power of positive line bundles by using peak section method. The first fewer coefficients of Bergman kernel asymptotic 
expansion were calculated by Lu~\cite{Lu} (see also~~\cite{HHL18},~\cite{Hsiao12}~\cite{Hsiao16},\cite{LuW},~\cite{MM06},~\cite{MM07}) and play an important role in K\"ahler geometry (see Donaldson~\cite{Don01}). On the other hand, the first coefficient at the diagonal of the expansion \eqref{e-gue201217yydI} is known, but the lower order terms of the expansion \eqref{e-gue201217yydI} were not yet known. Boutet de Monvel--Sj\"ostrand~\cite{BouSj76} (see also~\cite{Hsiao10}) showed that 
\begin{equation}\label{e-gue201221}
a_0(x,x)=\frac{1}{2\pi^{n+1}}\det\mathcal{L}_x,\ \ \mbox{at every $x\in X$},
\end{equation}
where $\det\mathcal{L}_x:=\mu_1(x)\mu_2(x)\cdots\mu_n(x)$, $\mu_j(x)$, $j=1,\ldots,n$, are the eigenvalues of the Levi form with respect to the given Hermitian metric on $\mathbb CTX$. If we take Levi metric on $\mathbb CTX$ (see \eqref{e-gue201221yydII}), then
\begin{equation}\label{e-gue201224yydq}
a_0(x,x)\equiv\frac{1}{2\pi^{n+1}}\ \ \mbox{on $D$}.
\end{equation}
It is a very natural question to calculate the lower order terms of the expansion \eqref{e-gue201217yydI}. The goal of this paper is to calculate $a_1(x,x)$ in \eqref{e-gue201217yydI}. It should be mentioned that the explicit formula of $a_1(x,x)$ has further applications in the study of CR Toeplitz quantization. 

We now formulate our main result. For some standard notations and terminology, we refer the reader to Sections~\ref{s-ssna},~\ref{s-gue201221yyd}. Let $(X,T^{1,0}X)$ be an orientable, compact strongly pseudoconvex CR manifold of dimension $2n+1$, $n\geq1$. Fix a global  non-vanishing $1$-form $\omega_0(x)\in\C^\infty(X,T^*X)$ so that 
\begin{equation}\label{e-gue201221yyd}
\begin{split}
&\mbox{$\omega_0(u)=0$, for every $u\in T^{1,0}X\oplus T^{0,1}X$},\\
&\mbox{$-\frac{1}{2i}d\omega_0|_{T^{1,0}X}$ is positive definite}. 
\end{split}
\end{equation}
For every $x\in X$, the Levi form at $x$ is the Hermitian quadratic form on $T^{1,0}_xX$ given by 
\begin{equation}\label{e-gue201221yydI}
\mbox{$\mathcal{L}_x(u, \ol v):=-\frac{1}{2i}\langle\,d\omega_0(x)\,,\,u\wedge\ol v\,\rangle$, for all $u, v\in T^{1,0}_xX$}. 
\end{equation}
Let $T\in\C^\infty(X,TX)$ be the global real vector field given by 
\begin{equation}\label{e-gue201222yyd}
\omega_0(T)\equiv -1,\ \ d\omega_0(T,\cdot)\equiv0. 
\end{equation}
The Levi form $\mathcal{L}_x$ induces a Hermitian metric (called Levi metric) $\langle\,\cdot\,|\,\cdot\,\rangle\,$ on $\mathbb CTX$ given by 
\begin{equation}\label{e-gue201221yydII}
\begin{split}
&\langle\,u\,|\,v\,\rangle=\mathcal{L}_x(u, \ol v),\ \ \mbox{for every $u, v\in T^{1,0}_xX$, $x\in X$},\\
&\langle\,\ol u\,|\,\ol v\,\rangle=\ol{\langle\,u\,|\,v\,\rangle},\ \ \mbox{for every $u, v\in T^{1,0}_xX$, $x\in X$},\\
&T^{1,0}X\perp T^{0,1}X,\ \ T\perp(T^{1,0}X\oplus T^{0,1}X),\\
&\langle\,T\,|\,T\,\rangle=1.
\end{split}
\end{equation}
Let $(\,\cdot\,|\,\cdot\,)$ be the $L^2$ inner product on $\Omega^{0,q}(X)$ induced by $\langle\,\cdot\,|\,\cdot\,\rangle$ and let $L^2_{(0,q)}(X)$ be the completion of $\Omega^{0,q}(X)$ with respect to 
$(\,\cdot\,|\,\cdot\,)$. We write $L^2(X):=L^2_{(0,0)}(X)$. Let 
\[\ddbar_b:\C^\infty(X)\To\Omega^{0,1}(X)\]
be the tangential Cauchy-Riemann operator (see \eqref{tangential Cauchy Riemann operator}) and we extend $\ddbar_b$ to $L^2$ space by 
\begin{equation}\label{e-gue201221ycds}
\begin{split}
&\ddbar_b: {\rm Dom\,}\ddbar_b\subset L^2(X)\To L^2_{(0,1)}(X),\\
&{\rm Dom\,}\ddbar_b:=\{u\in L^2(X);\, \ddbar_bu\in L^2_{(0,1)}(X)\}.
\end{split}
\end{equation} 
Let 
\[\Pi: L^2(X)\To{\rm Ker\,}\ddbar_b\]
be the orthogonal projection and let $\Pi(x,y)\in\mathscr D'(X\times X)$ be the distribution kernel of $\Pi$. 

Let $D\subset X$ be any open  open coordinate patch with local coordinates $x=(x_1,\ldots,x_{2n+1})$. For $b(x,y,t)\in S^n_{{\rm cl\,}}(D\times D\times\mathbb R_+)$, we write $b_0(x,y)$ and $b_1(x,y)$ to denote the first order term of $b$ and the second order term of $b$ respectively. Before we state our main result, it should be mentioned that the phase function $\phi$ in \eqref{Boutet-Sjostrand formula} is not unique and also the terms $a_0(x,y)$ and $a_1(x,y)$ are not unique. We can replace the phase function $\phi$ by $\hat\phi:=f\phi$, where $f$ is a smooth function with $f(x,x)=1$. Then, $\hat\phi$ satisfies \eqref{e-gue201217yyd} and $\phi$ and $\hat\phi$ are equivalent in the sense of Melin–Sj\"ostrand (see Melin–Sj\"ostrand~\cite[p. 172]{MS74}). 
When we change $\phi$ to $\hat\phi$ in \eqref{Boutet-Sjostrand formula}, the symbol $a(x,y,t)$ and  $a_0(x,y)$ and $a_1(x,y)$ will also be changed. Hence, $a_0(x,y)$ and $a_1(x,y)$ depend on the phase function and are not
unique. Even we fix $\phi$ in \eqref{Boutet-Sjostrand formula}, the functions $a_0(x,y)$ and $a_1(x,x)$ are not unique. For example, we can take 
\[\begin{split}
&\hat a_0(x,y):=a_0(x,y)+c(x,y)\phi(x,y),\ \ c(x,y)\in\C^\infty(D\times D),\\
&\hat a_1(x,y):=a_1(x,y)-nic(x,y). 
\end{split}\]
Set 
\[\mbox{$\hat a(x,y,t)\sim t^n\hat a_0(x,y)+t^{n-1}\hat a_1(x,y)+\sum^{+\infty}_{j=2}t^{n-j}a_j(x,y)$ in $S^n_{1,0}(D\times D\times\mathbb R_+)$}.\]
It is not difficult to see that 
\[\int_0^\infty e^{it\phi(x,y)}\hat a(x,y,t)dt\equiv \int_0^\infty e^{it\phi(x,y)}a(x,y,t)dt\mod\C^\infty(D\times D).\]
Hence, $a_0(x,y)$ and $a_1(x,x)$ are not unique.  

To  overcome the difficulty about the uniqueness for the symbols, we first look for some specific phase functions. 
Let $x=(x_1,\ldots,x_{2n+1})$ be local coordinates of $X$ defined on an open set $D$ of $X$ with $T=-\frac{\pr}{\pr x_{2n+1}}$ on $D$. By Malgrange preparation theorem~\cite[Theorem 7.5.5]{Hormander2003ALPDO1}, we have 
\[\phi(x,y)=f(x,y)(-x_{2n+1}+g(x',y))\ \ \mbox{on $D$},\]
where $\phi$ is as in \eqref{Boutet-Sjostrand formula}, $f, g\in\C^\infty(D\times D)$, $f(x,x)=1$, for every $x\in D$. Let 
\[\Phi:=-x_{2n+1}+g(x',y).\]
It is not difficult to see that $\Phi$ satisfies \eqref{e-gue201217yyd}, $\Phi$ and $\phi$ are equivalent in the sense of Melin–Sj\"ostrand. Moreover, we have 
\begin{equation}\label{e-gue201223yydy}
(T^2\Phi)(x,x)=0,\ \ \mbox{at every $x\in D$}. 
\end{equation}
Now we show that under
supplementary conditions, we have uniqueness (see Lemma~\ref{l-gue201224yyd} for a proof)

\begin{lemma}\label{l-gue201222yyd}
Let $D\subset X$ 
be any open coordinate patch with local coordinates $x=(x_1,\ldots,x_{2n+1})$.
Let $\phi_1, \phi_2\in\C^\infty(D\times D)$. Suppose that $\phi_1$, $\phi_2$ satisfy \eqref{e-gue201217yyd}, \eqref{e-gue201223yydy}, $\phi_1$ and $\phi$ are equivalent in the sense of Melin–Sj\"ostrand, $\phi_2$ and $\phi$ are equivalent in the sense of Melin–Sj\"ostrand, where $\phi$ is as 
in \eqref{Boutet-Sjostrand formula}. Suppose that 
\[\int^{+\infty}_0e^{i\phi_2(x,y)t}\alpha(x,y,t)dt\equiv\int^{+\infty}_0e^{i\phi_1(x,y)t}\beta(x,y,t)dt\mod\C^\infty(D\times D),\]
where $\alpha(x,y,t), \beta(x,y,t)\in S^n_{{\rm cl\,}}(D\times D\times\mathbb R_+)$ with 
\begin{equation}\label{e-gue201225yyd}
\begin{split}
&\alpha_0(x,x)=\beta_0(x,x),\ \ \mbox{for all $x\in D$},\\
&(T\alpha_0)(x,x)=(T\beta_0)(x,x)=0,\ \ \mbox{for all $x\in D$}.
\end{split}
\end{equation}
Then, we have 
\begin{equation}\label{e-gue201222yyda}
\alpha_1(x,x)=\beta_1(x,x),\ \ \mbox{for every $x\in D$}. 
\end{equation}
\end{lemma}

From Lemma~\ref{l-gue201222yyd}, we see that if we can choose the leading term in the expansion \eqref{e-gue201217yydI} satisfying \eqref{e-gue201225yyd} for all equivalent phase functions $\phi$ satisfying \eqref{e-gue201217yyd}, \eqref{e-gue201223yydy}, then the second coefficient of the expansion \eqref{e-gue201217yydI} at the diagonal is is uniquely determined. This is possible by Lemma~\ref{a_0 is independent of x_2n+1} below. Moreover, we will show in Lemma~\ref{a_0 is independent of x_2n+1} that we can take the  leading term in the expansion \eqref{e-gue201217yydI} as a function $\frac{1}{2\pi^{n+1}}+r(x,y)$, $r=O(|x-y|)$, $Tr=0$,  for all equivalent phase functions $\phi$ satisfying \eqref{e-gue201217yyd}, \eqref{e-gue201223yydy}, and hence the second coefficient of the expansion \eqref{e-gue201217yydI} at the diagonal is well-defined. 
The main result of this work is the following  

\begin{theorem}\label{main theorem}
With the notations and assumptions above, suppose that  $\overline{\partial}_b: {\rm Dom\,}\ddbar_b\subset L^2(X)\To L^2_{(0,1)}(X)$ has closed range. 
Let $D\subset X$ 
be any open coordinate patch with local coordinates $x=(x_1,\ldots,x_{2n+1})$. Let $\hat\phi\in\C^\infty(D\times D)$. Suppose that $\hat\phi$ satisfies \eqref{e-gue201217yyd}, \eqref{e-gue201223yydy} and $\phi$ and $\hat\phi$ are equivalent in the sense of Melin–Sj\"ostrand, where $\phi$ is as in \eqref{Boutet-Sjostrand formula}. Then, we can find $A(x,y,t)\in S^n_{{\rm cl\,}}(D\times D\times\mathbb R_+)$ with 
\begin{equation}\label{e-gue201221yyda}
\begin{split}
&\mbox{$A_0(x,x)=\frac{1}{2\pi^{n+1}}$, for every $(x,y)\in D\times D$}, \\
&TA_0=0\ \ \mbox{on $D$},
\end{split}
\end{equation}
\begin{equation}\label{e-gue201221yydb}
\mbox{$A_1(x,x)=\frac{1}{4\pi^{n+1}}R_{\mathrm{scal}}(x)$, for every $x\in D$}, 
\end{equation}
such that 
\begin{equation}\label{e-gue201221yydc}
\Pi(x,y)\equiv\int_0^\infty e^{it\hat\phi(x,y)}A(x,y,t)dt\mod\C^\infty(D\times D),
\end{equation}
where $R_{\mathrm{scal}}$ is the Tanaka--Webster scalar curvature on $X$ (see \eqref{Tanaka--Webster scalar curvature}).
\end{theorem}

\section{Preliminaries} \label{s:prelim}
\subsection{Standard notations and some background} \label{s-ssna}
We use the following notations through this article: $\mathbb N=\{1,2,\ldots\}$ is the set of natural numbers, $\mathbb N_0=\mathbb N\bigcup\{0\}$, $\mathbb R$ is the set of 
real numbers, $\overline{\mathbb R}_+=\{x\in\mathbb R;\, x\geq0\}$. 
We write $\alpha=(\alpha_1,\ldots,\alpha_n)\in\mathbb N^n_0$ 
if $\alpha_j\in\mathbb N_0$, 
$j=1,\ldots,n$. For $x=(x_1,\ldots,x_n)\in\mathbb R^n$, 
we write
\[
\begin{split}
&x^\alpha=x_1^{\alpha_1}\ldots x^{\alpha_n}_n,\\
& \pr_{x_j}=\frac{\pr}{\pr x_j}\,,\quad
\pr^\alpha_x=\pr^{\alpha_1}_{x_1}\ldots\pr^{\alpha_n}_{x_n}
=\frac{\pr^{|\alpha|}}{\pr x^\alpha}\,.
\end{split}
\]
Let $z=(z_1,\ldots,z_n)$, $z_j=x_{2j-1}+ix_{2j}$, $j=1,\ldots,n$, 
be coordinates of $\mathbb C^n$. We write
\[
\begin{split}
&z^\alpha=z_1^{\alpha_1}\ldots z^{\alpha_n}_n\,,
\quad\ol z^\alpha=\ol z_1^{\alpha_1}\ldots\ol z^{\alpha_n}_n\,,\\
&\pr_{z_j}=\frac{\pr}{\pr z_j}=
\frac{1}{2}\Big(\frac{\pr}{\pr x_{2j-1}}-i\frac{\pr}{\pr x_{2j}}\Big)\,,
\quad\pr_{\ol z_j}=\frac{\pr}{\pr\ol z_j}
=\frac{1}{2}\Big(\frac{\pr}{\pr x_{2j-1}}+i\frac{\pr}{\pr x_{2j}}\Big),\\
&\pr^\alpha_z=\pr^{\alpha_1}_{z_1}\ldots\pr^{\alpha_n}_{z_n}
=\frac{\pr^{|\alpha|}}{\pr z^\alpha}\,,\quad
\pr^\alpha_{\ol z}=\pr^{\alpha_1}_{\ol z_1}\ldots\pr^{\alpha_n}_{\ol z_n}
=\frac{\pr^{|\alpha|}}{\pr\ol z^\alpha}\,.
\end{split}
\]
For $j, s\in\mathbb Z$, set $\delta_{js}=1$ if $j=s$, 
$\delta_{js}=0$ if $j\neq s$.
 
Let $X$ be a $\C^\infty$ paracompact manifold.
We let $TX$ and $T^*X$ denote the tangent bundle of $X$
and the cotangent bundle of $X$ respectively.
The complexified tangent bundle of $X$ and the complexified cotangent bundle of $X$ will be denoted by $\mathbb CTX$
and $\mathbb C T^*X$, respectively. Write $\langle\,\cdot\,,\cdot\,\rangle$ to denote the pointwise
duality between $TX$ and $T^*X$.
We extend $\langle\,\cdot\,,\cdot\,\rangle$ bilinearly to $\mathbb CTX\times\mathbb C T^*X$.

Let $D\subset X$ be an open set . The spaces of distributions of $D$ and
smooth functions of $D$ will be denoted by $\mathscr D'(D)$ and $\C^\infty(D)$ respectively.
Let $\mathscr E'(D)$ be the subspace of $\mathscr D'(D)$ whose elements have compact support in $D$. 
Let $\C^\infty_c(D)$ be the subspace of $\C^\infty(D)$ whose elements have compact support in $D$.
Let $A: \C^\infty_c(D)\rightarrow \mathscr D'(D)$ be a continuous operator. We write $A(x, y)$ to denote the distribution kernel of $A$.
In this work, we will identify $A$ with $A(x,y)$. 
The following two statements are equivalent
\begin{enumerate}
\item $A$ is continuous: $\mathscr E'(D)\rightarrow\C^\infty(D)$,
\item $A(x,y)\in\C^\infty(D\times D)$.
\end{enumerate}
If $A$ satisfies (1) or (2), we say that $A$ is smoothing on $D$. Let
$A,B: \C^\infty_c(D)\rightarrow\mathscr D'(D)$ be continuous operators.
We write 
\begin{equation} \label{e-gue201114yyd}
\mbox{$A\equiv B$ on $D$} 
\end{equation}
or 
\begin{equation} \label{e-gue201114yydz}
A(x,y)\equiv B(x,y)\mod\C^\infty(D\times D)
\end{equation}
if $A-B$ is a smoothing operator. We sometimes will omit "on $D$" or "$\mod\C^\infty(D\times D)$" in \eqref{e-gue201114yyd} and \eqref{e-gue201114yydz} respectively. 
We say that $A$ is properly supported if the restrictions of the two projections 
$(x,y)\rightarrow x$, $(x,y)\rightarrow y$ to ${\rm Supp\,}K_A$
are proper. 

Let $X$ be a smooth orientable manifold of real dimension $2n+1$. Let $D$ be an open coordinate patch of $X$ with local coordinates $x$. We recall the following H\"ormander symbol space

\begin{definition}\label{d-gue201114yyd}
For $m\in\mathbb R$, $S^m_{1,0}(D\times D\times\mathbb{R}_+)$ 
is the space of all $a(x,y,t)\in\C^\infty(D\times D\times\mathbb{R}_+)$ 
such that for all compact $K\Subset D\times D$ and all $\alpha, 
\beta\in\mathbb N^{2n+1}_0$, $\gamma\in\mathbb N_0$, 
there is a constant $C_{\alpha,\beta,\gamma}>0$ such that 
\[|\partial^\alpha_x\partial^\beta_y\partial^\gamma_t a(x,y,t)|\leq 
C_{\alpha,\beta,\gamma}(1+|t|)^{m-|\gamma|},\ \ 
\mbox{for all $(x,y,t)\in K\times\mathbb R_+$, $t\geq1$}.\]
Put 
\[
S^{-\infty}(D\times D\times\mathbb{R}_+)
:=\bigcap_{m\in\mathbb R}S^m_{1,0}(D\times D\times\mathbb{R}_+).\]
Let $a_j\in S^{m_j}_{1,0}(D\times D\times\mathbb{R}_+)$, 
$j=0,1,2,\ldots$ with $m_j\rightarrow-\infty$, $j\rightarrow\infty$. 
Then there exists $a\in S^{m_0}_{1,0}(D\times D\times\mathbb{R}_+)$ 
unique modulo $S^{-\infty}$, such that 
$a-\sum^{k-1}_{j=0}a_j\in S^{m_k}_{1,0}(D\times D\times\mathbb{R}_+)$ 
for $k=1,2,\ldots$. 

If $a$ and $a_j$ have the properties above, we write $a\sim\sum^{\infty}_{j=0}a_j$ in 
$S^{m_0}_{1,0}(D\times D\times\mathbb{R}_+)$. We write
\begin{equation}  \label{e-gue201114yydI}
s(x, y, t)\in S^{m}_{{\rm cl\,}}(D\times D\times\mathbb{R}_+)
\end{equation}
if $s(x, y, t)\in S^{m}_{1,0}(D\times D\times\mathbb{R}_+)$ and 
\begin{equation}\label{e-gue201114yydII}
\begin{split}
&s(x, y, t)\sim\sum^\infty_{j=0}s_j(x, y)t^{m-j}\text{ in }S^{m}_{1, 0}
(D\times D\times\mathbb{R}_+)\,,\\
&s_j(x, y)\in\C^\infty(D\times D),\ j\in\mathbb N_0.
\end{split}\end{equation}
We sometimes omit "in $S^{m}_{1, 0}
(D\times D\times\mathbb{R}_+)$" in \eqref{e-gue201114yydII}. 
\end{definition} 

To calculate the first order term of Szeg\H{o} kernel expansion explicitly, we also need the following version of stationary phase formula~\cite[Theorem 7.7.5]{Hormander2003ALPDO1}

\begin{theorem}
\label{Hormander stationary phase formula}
Let $D\subset\mathbb{R}^n$ be an open set,  $K\subset D$ be a compact set, $F\in \mathscr{C}^\infty(D)$, $\mathrm{Im}F\geq 0$ in $D$. Assume
\[
\mathrm{Im}F(0)=0,~F'(0)=0,~\det~F''(0)\neq 0,~F'\neq 0~\text{in}~K\setminus\{0\}.
\]
Let $u\in\C^\infty_c(D)$, ${\rm Supp\,}u\subset K$, 
Then, for any $k>0$,
\[
\left|\int e^{ikF(z)}u(x)dx-e^{ikF(0)}\det\left(\frac{kF''(0)}{2\pi i}\right)^{-\frac{1}{2}}\sum_{j<N}k^{-j}P_j u\right|\leq C k^{-N}\sum_{|\alpha|\leq N}\sup_K|\partial_x^\alpha u|.
\]
Here, $C$ is a bounded constant when $F$ is bounded in $\mathscr{C}^\infty(D)$, $\frac{|x|}{|F'(x)|}$ has a uniform bound and
\[
P_ju:=\sum_{v-\mu=j}\sum_{2v\geq 3\mu}i^{-j}2^{-v}\langle F''(0)^{-1}D,D\rangle^v\frac{(h^\mu u)(0)}{v!\mu!}.
\]
$h(x):=F(x)-F(0)-\frac{1}{2}\langle F''(0)x,x\rangle$,
$D=
\begin{pmatrix}
-i\partial_{x_1}\\
\vdots\\
-i\partial_{x_n}
\end{pmatrix}$.
\end{theorem}

\subsection{Abstract CR manifolds}\label{s-gue201221yyd}
Let $X$ be a smooth orientable manifold of real dimension $2n+1$ (at least three), we say $X$ is a Cauchy--Riemann manifold (CR manifold for short) if there is a subbundle
$T^{1,0}X\subset\mathbb{C}TX$, such that
\begin{enumerate}
\item $\dim_{\mathbb{C}}T^{1,0}_{p}X=n$ for any $p\in X$.
\item $T^{1,0}_p X\cap T^{0,1}_p X=\{0\}$ for any $p\in X$, where $T^{0,1}_p X:=\overline{T^{1,0}_p X}$.
\item For $V_1, V_2\in \C^\infty(X,T^{1,0}X)$, then $[V_1,V_2]\in\C(X,T^{1,0}X)$, where 
$[\cdot,\cdot]$ stands for the Lie bracket between vector fields. 
\end{enumerate}
For such subbundle $T^{1,0}X$, we call it a CR structure of the CR manifold $X$. Let $(X,T^{1,0}X)$ be a smooth orientable CR manifold of 
dimension $2n+1$. 
For dimension reason and the assumption that $X$ is orientable, there is a global real non-vanishing one form  $\omega_0(x)$ such that 
\[\langle\,\omega_0(x)\,,u\,\rangle=0,\ \ \mbox{for every $u\in T^{1,0}_xX\oplus T^{0,1}_xX$}.\]
We define the Levi form (with respect to $\omega_0$) which is a globally defined $(1,1)$-form, by
\begin{equation}
    \label{Levi form}
    \mathcal{L}_x({u},\overline{v}):=\frac{1}{2i}\left\langle\omega_0(x),\left[\mathring{u},\overline{\mathring{v}}\right](x)\right\rangle,
\end{equation}
where $\mathring{u},\mathring{v}\in\C^\infty(X,T^{1,0} X)$ such that $\mathring{u}(x)=u\in T_x^{1,0}X$ and $\mathring{v}(x)=v\in T_x^{1,0}X$. Note that by Cartan magic formula, we can also express the Levi form by 
\begin{equation}
    \label{Levi form by d omega_0}
    \mathcal{L}_x(u,\overline{v})=-\frac{1}{2i}\left\langle d\omega_0(x),u\wedge\overline{v}\right\rangle,~u,v\in T^{1,0}_x X.
\end{equation}
In other words,
$$
\mathcal{L}_x:=\left.-\frac{1}{2i}d\omega_0(x)\right|_{T_x^{1,0}X}.
$$
\begin{definition}
We say a CR manifold $X$ is strongly pseudoconvex if we can find $\omega_0$ so that $\mathcal{L}_x$ is positive definite for all $x\in X$.
\end{definition}

From now on, we assume that $X$ is strongly pseudoconvex and we fix $\omega_0$ so that $\mathcal{L}_x$ is positive definite for all $x\in X$. 
Let $T\in\C^\infty(X,TX)$ be the global real vector field given by 
\begin{equation}\label{e-gue201223yyd}
\omega_0(T)\equiv -1,\ \ d\omega_0(T,\cdot)\equiv0. 
\end{equation}
The Levi form $\mathcal{L}_x$ induces a Hermitian metric (called Levi metric) $\langle\,\cdot\,|\,\cdot\,\rangle\,$ on $\mathbb CTX$ given by 
\begin{equation}\label{e-gue201223yydI}
\begin{split}
&\langle\,u\,|\,v\,\rangle=\mathcal{L}_x(u, \ol v),\ \ \mbox{for every $u, v\in T^{1,0}_xX$, $x\in X$},\\
&\langle\,\ol u\,|\,\ol v\,\rangle=\ol{\langle\,u\,|\,v\,\rangle},\ \ \mbox{for every $u, v\in T^{1,0}_xX$, $x\in X$},\\
&T^{1,0}X\perp T^{0,1}X,\ \ T\perp(T^{1,0}X\oplus T^{0,1}X),\\
&\langle\,T\,|\,T\,\rangle=1.
\end{split}
\end{equation}

Let $\Gamma:\mathbb{C}T_x X\to \mathbb{C}T_x^*X$ be the anti-linear map given by $\langle u|v\rangle=\langle u,\Gamma v\rangle$ for $u,v\in\mathbb{C}T_x X$, then we can take the induced Hermitian metric on $\mathbb{C}T^*X$ by $\langle u|v\rangle:=\langle\Gamma^{-1}v|\Gamma^{-1}u\rangle$ for $u,v\in\mathbb{C}T^*_x X$. Put 
$$
T^{*{1,0}}X:=\Gamma(T^{1,0}X)={(T^{0,1}X\oplus\mathbb{C}T)}^{\perp}\subset\mathbb{C}T^*X,~T^{*{0,1}}X:=\overline{T^{*{1,0}}X}.
$$
Take the Hermitian metric on $\Lambda^r(\mathbb{C}T^*X)$ by
$$
\langle u_1\wedge\cdots \wedge u_r|v_1\wedge\cdots v_r\rangle=\det\left(\left(\langle u_j|u_k\rangle\right)_{j,k=1}^r\right),~\text{where}~u_j,v_k\in\mathbb{C}T^*X~,j,k=1,\cdots,r.
$$
For every $q\in\{0,1,\ldots,n\}$, the bundle of $(0,q)$ forms of $X$ is given by $T^{*0,q}X:=\Lambda^q(T^{*0,1}X)$ 
and let $\Omega^{0,q}(X)$ be the space of smooth $(0,q)$ forms on $X$. Let 
$$
\pi^{(0,q)}:\Lambda^q(\mathbb{C}T^*X)\to T^{*0,q}X
$$
be the orthogonal projection with respect to $\langle\,\cdot\,|\,\cdot\,\rangle$. The tangential Cauchy--Riemann operator is defined by
\begin{equation}
    \label{tangential Cauchy Riemann operator}
    \overline{\partial}_b:=\pi^{(0,q+1)}\circ d: \Omega^{0,q}(X)\To\Omega^{0,q+1}(X).
\end{equation}
By Cartan magic formula, we can check that 
$$
\overline{\partial}_b^2=0.
$$
Take the $L^2$-inner product $(\cdot|\cdot)$ on $\Omega^{0,q}(X)$ induced by $\langle\cdot|\cdot\rangle$ via
$$
(f|g):=\int_X\langle f|g\rangle dV_X,~f,g\in\Omega^{0,q}(X),
$$
where $dV_X$ is the volume form with expression 
$$
dV_X(x)=\sqrt{\det\left(\left\langle\frac{\partial}{\partial x_j}\middle|\frac{\partial}{\partial x_k}\right\rangle\right)_{j,k=1}^{2n+1}}dx_1\wedge\cdots \wedge dx_{2n+1}
$$
in local coordinates $(x_1,\cdots,x_{2n+1})$. Let $(\,\cdot\,|\,\cdot\,)$ be the $L^2$ inner product on $\Omega^{0,q}(X)$ induced by $\langle\,\cdot\,|\,\cdot\,\rangle$ and let $L^2_{(0,q)}(X)$ be the completion of $\Omega^{0,q}(X)$ with respect to 
$(\,\cdot\,|\,\cdot\,)$. We write $L^2(X):=L^2_{(0,0)}(X)$.  We extend $\ddbar_b$ to $L^2$ space by 
\begin{equation}\label{e-gue201223yydu}
\begin{split}
&\ddbar_b: {\rm Dom\,}\ddbar_b\subset L^2(X)\To L^2_{(0,1)}(X),\\
&{\rm Dom\,}\ddbar_b:=\{u\in L^2(X);\, \ddbar_bu\in L^2_{(0,1)}(X)\}.
\end{split}
\end{equation} 

\begin{definition}\label{d-gue201223yyd}
The orthogonal projection
$$
\Pi: L^2(X)\To{\rm Ker\,}\ddbar_b:=\{u\in\mathrm{Dom}\ddbar_b;\, \ddbar_bu=0\}
$$ is called the Szeg\H{o} projection, and we call
its distributional kernel $\Pi(x,y)$ the Szeg\H{o} kernel.
\end{definition}

\subsection{Pseudohermitian geometry on strongly pseudovoncex CR manifolds}

We will use the same notations as before. We call 
$$
HX:=\mathrm{Re}\left(T^{1,0}X\oplus T^{0,1}X\right)
$$
the contact structure of $X$, and let $J$ be the complex structure on $HX$ so that $T^{1,0}X$ is the eigenspace of $J$ corresponding to the eigenvalue $i$. Let 
\[\theta_0:=-\omega_0.\]
The following is well-known:
\begin{proposition}\label{p-gue201228yyds}[\cite{Tanaka1975SPCDG}*{Proposition 3.1}]
With the same notations and assumptions, there exists an unique affine connection, called Tanaka--Webster connection, \[
\nabla:=\nabla^{\theta_0}: \C^\infty(X,TX)\to \C^\infty(X,T^*X\otimes TX)
\]
such that
\begin{enumerate}
    \item $\nabla_U \C^\infty(X,HX)\subset \C^\infty(X,HX)$ for $U\in\C^\infty(X,TX)$. 
    \item $\nabla T=\nabla J=\nabla d\theta_0=0$.
    \item The torsion $\tau$ of $\nabla$ satisfies:
  ${\tau}(U, V)=d\theta_0(U, V)T$, ${\tau}(T, JU)=-J{\tau}(T, U)$,
  $U, V\in\C^\infty(X, HX).$
\end{enumerate}
\end{proposition} 

Recall that $\nabla J\in\C^\infty(X, T^* X\otimes\mathscr L(HX,HX))$,
$\nabla d\theta_0\in\C^\infty(T^*X\otimes\Lambda^2(\mathbb C T^* X))$
are defined by $(\nabla_UJ)W=\nabla_U(JW)-J\nabla_UW$ and
$\nabla_Ud\theta_0(W, V)=Ud\theta_0(W, V)-d\theta_0(\nabla_UW, V)-
d\theta_0(W, \nabla_UV)$ for $U\in\C^\infty(X,TX), W, V\in\C^\infty(X,HX)$.
Moreover, $\nabla J=0$ and $\nabla d\omega_0=0$
imply that the Tanaka-Webster connection is compatible with the Levi metric.
By definition, the torsion of $\nabla$ is given by
$\tau (W, U)=\nabla_WU-\nabla_UW-[W, U]$ for $U, V\in\C^\infty(X,TX)$ and
$\tau(T, U)$ for $U\in\C^\infty(X,HX)$ is called pseudohermitian torsion.

Let $\{L_\alpha\}_{\alpha=1}^n$ be a local frame of $T^{1,0}X$ and $\{\theta^\alpha\}_{\alpha=1}^n$ be the dual frame of $\{L_\alpha\}_{\alpha=1}^n$. We use the notations $Z_{\overline{\alpha}}:=\overline{L_\alpha}$ and $\theta^{\overline{\alpha}}=\overline{\theta^\alpha}$. Write
\[
\nabla L_\alpha=\omega_\alpha^\beta\otimes L_\beta,~\nabla L_{\overline{\alpha}}=\omega_{\overline{\alpha}}^{\overline{\beta}}\otimes L_{\overline{\beta}},~\text{and recall that}~\nabla T=0.
\]
We call $\omega_\alpha^\beta$ the connection one form of Tanaka--Webster connection with respect to the frame $\{L_\alpha\}_{\alpha=1}^n$. We denote $\Theta_\alpha^\beta$ the Tanaka--Webster curvature two form, and it is known that
\[
\Theta_\alpha^\beta=d\omega_\alpha^\beta-\omega_\alpha^\gamma\wedge\omega_\gamma^\beta.
\]
By direct computation, we also have
\[
\Theta_\alpha^\beta=R_{\alpha j\overline{k}}^\beta\theta^j\wedge\theta^{\overline{k}}+A_{\alpha jk}^\beta\theta^j\wedge\theta^k+B_{\alpha jk}^\beta\theta^{\overline{j}}\wedge\theta^{\overline{k}}+C_0\wedge\theta_0,
\]
where $C_0$ is an one form. The term $R_{\alpha j\overline{k}}^\beta$ is called the pseudohermitian curvature tensor and the trace
\[
R_{\alpha\overline{k}}:=\sum_{j=1}^nR_{\alpha j\overline{k}}^j
\]
is called the pseudohermitian Ricci curvature. Also, write
\[
d\theta_0=ig_{\alpha\overline{\beta}}\theta^\alpha\wedge\theta^{\overline{\beta}}
\]
and $g^{\overline{c}d}$ be the inverse matrix of $g_{a\overline{b}}$, then the Tanaka--Webster scalar curvature $R$ with respect to the pseudohermitian structure $\theta_0$ is given by
\begin{equation}
\label{Tanaka--Webster scalar curvature}
    R_{\mathrm{scal}}:=g^{\ol k\alpha}R_{\alpha\overline{k}}.
\end{equation}

\section{The calculation of the second coefficient of the Szeg\H{o} kernel asymptotic expansion}

In this section, we will prove Theorem~\ref{main theorem}. 
We will first show that how to select a suitable phase function and the leading term of the Szeg\H{o} kernel asymptotic expansion such that the second order term of the Szeg\H{o} kernel asymptotic expansion is well-defined.
\subsection{Uniqueness of the coefficients} 

We first recall some facts from the theory of oscillatory integral and distributions: Notice that $\int_0^\infty e^{-tx}dt=x^{-1}$, for ${\rm Re\,}x>0$.
Also notice that by partial integration and dominated convergence theorem,
\begin{align*}
    \int_0^1\frac{1-e^{-t}}{t}dt+\int_1^\infty e^{-t}t^{-1}dt
    &=\int_0^\infty e^{-t}\log tdt\\
    &=\lim_{m\to\infty}\left(\int_0^m \left(1-\frac{t}{m} \right)^{m-1}\log{t}dt\right)\\
&=\lim_{m\to\infty}\left(m\int_0^1 s^{m-1} \log{(m(1-s))}ds\right)\\
&=\lim_{m\to\infty}\left(m\log{m}\int_0^1 s^{m-1}ds+m\int_0^1 s^{m-1}\log{(1-s)}ds\right)\\
&=\lim_{m\to\infty}\left(\log{m}-m\int_0^1 \sum_{k=1}^{\infty}\frac{s^{k+m-1}}{k}ds\right)\\
&=\lim_{m\to\infty}\left(\log{m}-m\sum_{k=1}^{\infty}\int_0^1\frac{s^{k+m-1}}{k}ds\right)\\
&=\lim_{m\to\infty}\left(\log{m}-m\sum_{k=1}^{\infty}\frac{1}{k(k+m)}\right)\\
&=\lim_{m\to\infty}\left(\log{m}-\sum_{k=1}^{\infty}\left(\frac{1}{k}-\frac{1}{m+k} \right)\right)\\
&=\lim_{m\to\infty}\left(\log{m}-\sum_{k=1}^{m}\frac{1}{k}\right)
\end{align*}
Accordingly, we can check that by partial integration, if $x\neq 0$ and $\mathrm{Re}~x\geq 0$, then
\begin{equation}
    \label{Gamma function}
    \mathrm{P.V.}\int_0^\infty e^{-tx}t^mdt=
\begin{cases}
m!x^{-m-1}:~m\in\mathbb{Z},~m\geq 0\\
\frac{(-1)^m}{(-m-1)!}x^{-m-1}\left(\log x+\gamma-\sum_{j=0}^{-m-1}\frac{1}{j}\right):m\in\mathbb{Z},~m<0
\end{cases}
\end{equation}
where $\gamma:=\lim_{m\to\infty}\left(\sum_{j=1}^m\frac{1}{j}-\log m\right)$ is the Euler constant. On the other hand, by choosing a suitable contour, we can check that for a smooth function $f\in\C^\infty_c(\mathbb R^n)$, ${\rm Im\,}f\geq0$, and any $\epsilon>0$,
\begin{equation}
\label{principal value distribution}
    \frac{1}{(-i)(f(x)+i\epsilon)}=\int_0^\infty e^{i(f(x)+i\epsilon)t}dt;
\end{equation}
and if $df(x)\neq 0$ if ${\rm Im\,}f(x)=0$, with convergence in $\mathscr{D}'(\mathbb R^n)$, we define
\[
\frac{1}{(-i)(f(x)+i0)}:=\lim_{\epsilon\to 0}\frac{1}{(-i)(f(x)+i\epsilon)},
\]
and we can check that
\[
\frac{1}{(-i)(f(x)+i0)}=\int_0^\infty e^{itf(x)}dt
\]
in the sense of oscillatory integral. Moreover, we have:

\begin{lemma}
\label{lemma of osc}
Let $D\subset\mathbb{R}^n$ be a small enough open set near $0$. Assume 
$$
F(x)\in \mathscr{C}^\infty(D),F(0)=0,~\mathrm{Im}F\geq0,~dF\neq 0\ \ \mbox{if ${\rm Im\,}F=0$},
$$
and
$$
G(x)\in \mathscr{C}^\infty(D),~G(0)\neq 0,~\mathrm{Im}(FG)\geq 0,~d(FG)\neq 0\ \ \mbox{if ${\rm Im\,}(FG)=0$}.
$$
Let $m\in\mathbb N_0$. 
Then, in the sense of oscillatory integral,
\[
\int_0^\infty e^{itG(x)F(x)}t^mdt\equiv\int_0^\infty e^{itF(x)}\frac{t^m}{G(x)^{m+1}}dt~\text{mod}~\mathscr{C}^\infty(D).
\]
\end{lemma}
\begin{proof}
First of all, by continuity, we may assume that on $D$
\begin{equation}
\label{norm assumption for F and G}
    |G|\geq\frac{1}{2}|G(0)|>0.
\end{equation}
From the construction of oscillatory integral \cite{Hormander2003ALPDO1}*{Theorem 7.8.2}, in the sense of distribution,
\begin{align*}
    \int_0^\infty e^{itG(x)F(x)}t^mdt
&=\lim_{\epsilon\to 0}\int_0^\infty e^{it(G(x)F(x)+i\epsilon)}t^mdt\\
&=\lim_{\epsilon\to 0}\frac{m!}{(-iG(x)F(x)+\epsilon)^{m+1}}\\
&=\frac{m!}{G(x)^{m+1}}\lim_{\epsilon\to 0}\frac{1}{(-iF(x)+\frac{\epsilon}{G(x)})^{m+1}}\\
&=\frac{m!}{(-iG(x))^{m+1}}\frac{1}{(F(x)+i0)^{m+1}}\\
&=\int_0^\infty e^{itF(x)}\frac{t^m}{G(x)^{m+1}}dt.
\end{align*}
\end{proof}

We first need 
\begin{lemma}\label{l-gue201224yyd} 
Fix $p\in X$. 
Let $D\subset X$ 
be any open coordinate patch of $X$ with local coordinates $x=(x_1,\ldots,x_{2n+1})$, $p\in D$. 
Let $\phi_1, \phi_2\in\C^\infty(D\times D)$. Suppose that $\phi_1$, $\phi_2$ satisfy \eqref{e-gue201217yyd}, $\phi_1$ and $\phi$ are equivalent in the sense of Melin–Sj\"ostrand, $\phi_2$ and $\phi$ are equivalent in the sense of Melin–Sj\"ostrand, where $\phi$ is as in \eqref{Boutet-Sjostrand formula}, and 
\begin{equation}\label{e-gue201224yyd}
(T^2\phi_1)(p,p)=(T^2\phi_2)(p,p)=0.
\end{equation}
 Suppose that 
\begin{equation}\label{phi_alpha and phi_beta are equivalent}
\int^{+\infty}_0e^{i\phi_2(x,y)t}\alpha(x,y,t)dt\equiv\int^{+\infty}_0e^{i\phi_1(x,y)t}\beta(x,y,t)dt\mod\C^\infty(D\times D),
\end{equation}
where $\alpha(x,y,t), \beta(x,y,t)\in S^n_{{\rm cl\,}}(D\times D\times\mathbb R_+)$ with 
\begin{equation}\label{e-gue201225yyda}
\begin{split}
&\alpha_0(p,p)=\beta_0(p,p),\\
&(T\alpha_0)(p,p)=(T\beta_0)(p,p)=0.
\end{split}
\end{equation}
 Then, we have 
\begin{equation}\label{e-gue201224yydI}
\alpha_1(p,p)=\beta_1(p,p).
\end{equation}
\end{lemma}

\begin{proof} 
We take local coordinates $x=(x_1,\ldots,x_{2n+1})$ of $X$ such that 
\begin{equation*}
    T=-\frac{\partial}{\partial x_{2n+1}}.
\end{equation*}
As in \cite[Section 8]{HM16}, we have 
\begin{equation*}
    \phi_2(x,y)=f(x,y)\phi_1(x,y)+O(|x-y|^\infty),
\end{equation*}
for some $f(x,y)\in\C^\infty(D\times D)$, and we may assume that
\begin{equation}\label{e-gue201225yydb}
    \phi_2(x,y)=f(x,y)\phi_1(x,y).
\end{equation}
From \eqref{e-gue201224yyd} and \eqref{e-gue201225yydb}, we can check that 
\begin{equation}\label{e-gue201225yydc}
f(x,x)=1,\ \ \frac{\pr f}{\pr x_{2n+1}}(0,0)=0,
\end{equation}
and hence 
\begin{equation}
\label{f=1+O|x_2n+1|^2}
    f\left((0,x_{2n+1}),0\right)=1+O\left(|x_{2n+1}|^2\right). 
\end{equation}
Applying Lemma \ref{lemma of osc} and (\ref{Gamma function}) to oscillatory integrals (\ref{phi_alpha and phi_beta are equivalent}), we see that
\begin{equation}
    \frac{\sum_{j=0}^n(n-j)!\alpha_j(x,y)(-i\phi_1 f)^j(x,y)~\mathrm{mod}~
    \phi_1^{n+1}}{f^{n+1}(x,y)(-i(\phi_1(x,y)+i0))^{n+1}}\equiv\frac{\sum_{j=0}^n (n-j)!\beta_j(x,y)(-i\phi_1)^j(x,y)~\mathrm{mod}~
    \phi_1^{n+1}}{(-i(\phi_1(x,y)+i0))^{n+1}}
\end{equation}
up to some log term singularities. In particular, when $x\neq y$, 
\begin{equation}\label{e-gue201228yydVII}
    \sum_{j=0}^n(n-j)!\alpha_j(x,y)(-i\phi_1 f)^j(x,y)=f^{n+1}(x,y)\sum_{j=0}^n (n-j)!\beta_j(x,y)(-i\phi_1)^j(x,y)+(-if\phi_1)^{n+1}(x,y)S(x,y).
\end{equation}
for some $S\in\C^\infty(D\times D)$. Now, take $x=(0,x_{2n+1})$ and $y=0$ in the above equation, then 
from \eqref{e-gue201225yyda}, \eqref{e-gue201225yydc} and \eqref{f=1+O|x_2n+1|^2}, it is straightforward to check that 
\begin{equation*}
(\alpha_1-\beta_1)((0,x_{2n+1}),0)=O(|x_{2n+1}|).
\end{equation*}
By taking $x_{2n+1}\to 0$, we see that
\begin{equation*}
\alpha_1(0,0)=\beta_1(0,0).
\end{equation*}
\end{proof} 

We need 

\begin{lemma}
\label{a_0 is independent of x_2n+1}
With the notations and assumptions used in Theorem~\ref{Boutet-Sjostrand-Hsiao formula}, we can take $a(x,y,t)$ in \eqref{Boutet-Sjostrand formula} so that 
\begin{equation}\label{e-gue201224yydo}
\begin{split}
&Ta_0=0\ \ \mbox{on $D$},\\
&a_0(x,x)=\frac{1}{2\pi^{n+1}}\ \ \mbox{for every $x\in D$}. 
\end{split}
\end{equation}
\end{lemma}

\begin{proof}
Take local coordinates $x=(x_1,\ldots,x_{2n+1})$ of $X$ so that $T=-\frac{\pr}{\pr x_{2n+1}}$ on $D$. By the relation
$$
\mbox{$d_x\phi(x,x)=-\omega_0(x)$, for every $x\in D$}, 
$$
we know that $\frac{\partial \phi}{\partial x_{2n+1}}(x,x)=-1\neq 0$. Thus, applying the Malgrange preparation theorem~\cite[Theorem 7.5.5]{Hormander2003ALPDO1}, we have
\begin{equation}
\label{phi=f_1(-x_{2n+1}+g_1(x',y))}
    \phi(x,y)=f_1(x,y)(-x_{2n+1}+g_1(x',y))
\end{equation}
for some smooth functions $f_1(x,y)$, $g_1(x',y)$ with $f_1(x,x)=1$, for every $x\in D$, where $x'=(x_1,\ldots,x_{2n})$. Note that $g_1$ is independent of $x_{2n+1}$. 
By Taylor formula, we have
\begin{equation}\label{phi=f_1(x_2n+1+g_1(x',y))}
    a_0(x,y)=\tilde{a}_0(x,y)=\tilde{a}_0((x',g_1(x',y)),y)+\left(-x_{2n+1}+g_1(x',y)\right)R(x,y),
\end{equation}
where $\tilde{a}_0$ denotes an almost analytic extension of $a_0$ with respect to the real variable $x_{2n+1}$, $R(x,y)\in\C^\infty(D\times D)$. Let
\begin{equation}\label{e-gue201227yyd}
    \hat{a}_0(x',y):=\tilde{a}_0((x',g_1(x',y)),y)\in \C^\infty(D\times D).
\end{equation}
Then, by using integration by parts, we get 
\begin{align*}
    \int_0^\infty e^{it\phi(x,y)}a_0(x,y)t^ndt
    &=\int_0^\infty e^{it\phi(x,y)}\hat{a}_0(x',y)t^ndt-i\int_0^\infty \frac{d}{dt}\left(e^{itf_1(x,y)(-x_{2n+1}+g_1(x',y))}\right)\frac{R(x,y)}{f_1(x,y)}t^ndt\\
    &=\int_0^\infty e^{it\phi(x,y)}\hat{a}_0(x',y)t^ndt+i\int_0^\infty e^{it\phi(x,y)}\frac{R(x,y)}{f_1(x,y)}nt^{n-1}dt.
\end{align*}
Hence, we can replace $a_0(x,y)$ by $\hat{a}_0(x',y)$ in \eqref{Boutet-Sjostrand formula} and by \eqref{e-gue201224yydq}, we see that 
\[\hat{a}_0(x,x)=\frac{1}{2\pi^{n+1}},\ \ \mbox{for every $x\in D$}. \]
The lemma follows.
\end{proof}

From now on, we assume that $a_0(x,y)$ satisfies \eqref{e-gue201224yydo}. We have 

\begin{theorem}
\label{a_0 is a constant}
With the notations and assumptions used in Theorem~\ref{Boutet-Sjostrand-Hsiao formula}, 
we assume that 
\begin{equation}\label{e-gue201226yyd}
\begin{split}
&\mbox{$\ddbar_{b,x}(\phi(x,y))$ vanishes to infinite order at $x=y$},\\
&\mbox{$\ddbar_{b,y}(-\ol \phi(y,x))$ vanishes to infinite order at $x=y$}.
\end{split}
\end{equation}
We have
\begin{equation}\label{e-gue201224yydp}
    a_0(x,y)-\frac{1}{2\pi^{n+1}}=O(|x-y|^N),\ \ \mbox{for every $(x,y)\in D\times D$, for every $N\in\mathbb N$}. 
\end{equation}
\end{theorem} 

\begin{proof}
Fix $p\in X$, let $x=(x_1,\ldots,x_{2n+1})$ be local coordinates of $X$ defined on $D$ with $x(p)=0$, $T=-\frac{\pr}{\pr x_{2n+1}}$ and 
\begin{equation}\label{e-gue201225yydx}
\begin{split}
&T^{1,0}X={\rm span\,}\{\ol L_j;\, j=1,\ldots,n\},\\
&\ol L_j=\frac{\pr}{\pr\ol z_j}+O(|x|),\ \ \frac{\pr}{\pr\ol z_j}=\frac{1}{2}(\frac{\pr}{x_{2j-1}}+i\frac{\pr}{\pr x_{2j}}),\ \ j=1,\ldots,n.
\end{split}
\end{equation}
From $\ddbar_b\Pi=0$ and $\ddbar_b\phi$ vanishes to infinite order at $x=y$, applying Malgrange preparation theorem \cite{Hormander2003ALPDO1}*{Theorem 7.5.6}, partial integration, Lemma \ref{lemma of osc}, Melin--Sj\"ostrand \cite{MS74}*{p. 172} and \cite{HM16}*{Section 8}, it is not difficult to check that 
\begin{equation}\label{e-gue201225yydy}
\begin{split}
&(\ol L_ja_0)(x,y)=h_j(x,y)(-x_{2n+1}+g_1(x',y))+O(|x-y|^N),\\
&\mbox{for every $N\in\mathbb N$},\ \ h_j\in \C^\infty(D\times D),\ \ j=1,\ldots,n,
\end{split}
\end{equation}
where $g_1\in\C^\infty(D\times D)$ is as in \eqref{phi=f_1(-x_{2n+1}+g_1(x',y))}. We claim that 
\begin{equation}\label{e-gue201226yydb}
\mbox{$|a_0(x,y)-\frac{1}{2\pi^{n+1}}|=O(|(x,y)|^N)$, for every $N\in\mathbb N_0$}. 
\end{equation}
It is clear that \eqref{e-gue201226yydb} holds for $N=0$. Suppose that \eqref{e-gue201226yydb} holds for $N=N_0$, $N_0\in\mathbb N_0$. 
We are going to prove that \eqref{e-gue201226yydb} 
holds for $N=N_0+1$. From \eqref{e-gue201225yydx}, \eqref{e-gue201225yydy}, \eqref{e-gue201226yydb} and induction assumption, we have 
\begin{equation}\label{e-gue201225ycdy}
\frac{\pr a_0}{\pr\ol z_j}(x',y)=O(|(x,y)|^{N_0}),\ \ j=1,\ldots,n.
\end{equation}
Similarly, we have 
\begin{equation}\label{e-gue201225ycdz}
\frac{\pr a_0}{\pr w_j}(x',y)=O(|(x,y)|^{N_0}),\ \ j=1,\ldots,n,
\end{equation}
where $\frac{\pr}{\pr w_j}=\frac{1}{2}(\frac{\pr}{\pr y_{2j-1}}-i\frac{\pr}{\pr y_{2j}})$, $j=1,\ldots,n$. 

Fix $j\in\{1,\ldots,n\}$ and fix $\alpha, \beta\in\mathbb N_0$, $\alpha+\beta=N_0$. 
From $a_0(x,x)=\frac{1}{2\pi^{n+1}}$, we have 
\begin{equation}\label{e-gue201226yydg}
\Bigr(\bigr((\frac{\pr}{\pr z_j}+\frac{\pr}{\pr  w_j})^\alpha(\frac{\pr}{\pr \ol z_j}+\frac{\pr}{\pr\ol w_j})^\beta\bigr)a_0\Bigr)(x,x)=0. 
\end{equation}
From \eqref{e-gue201226yydg}, we have 
\begin{equation}\label{e-gue201226yydh}
\begin{split}
&\Bigr((\frac{\pr}{\pr z_j})^{\alpha}(\frac{\pr}{\pr \ol w_j})^{\beta}a_0\Bigr)(0,0)\\
&=\sum_{\alpha_1, \alpha_2, \beta_1,\beta_2\in\mathbb N_0, \alpha_1+\alpha_2=\alpha, \beta_1+\beta_2=\beta, \alpha_2+\beta_1>0} c_{\alpha_1,\alpha_2,\beta_1,\beta_2}\Bigr((\frac{\pr}{\pr z_j})^{\alpha_1}(\frac{\pr}{\pr w_j})^{\alpha_2}(\frac{\pr}{\pr\ol z_j})^{\beta_1}(\frac{\pr}{\pr \ol w_j})^{\beta_2}a_0\Bigr)(0,0),
\end{split}
\end{equation}
where $c_{\alpha_1,\alpha_2,\beta_1,\beta_2}$ is a constant, for every $\alpha_1, \alpha_2, \beta_1,\beta_2\in\mathbb N_0$, $\alpha_1+\alpha_2=\alpha$, $\beta_1+\beta_2=\beta$, $\alpha_2+\beta_1>0$. Since $\alpha_2+\beta_1>0$, from \eqref{e-gue201225ycdy} and \eqref{e-gue201225ycdz}, we get 
\[\begin{split}
&\Bigr((\frac{\pr}{\pr z_j})^{\alpha_1}(\frac{\pr}{\pr w_j})^{\alpha_2}(\frac{\pr}{\pr\ol z_j})^{\beta_1}(\frac{\pr}{\pr \ol w_j})^{\beta_2}a_0\Bigr)(0,0)=0,\\
&\mbox{for every 
$\alpha_1, \alpha_2, \beta_1,\beta_2\in\mathbb N_0$, $\alpha_1+\alpha_2=\alpha$, $\beta_1+\beta_2=\beta$, $\alpha_2+\beta_1>0$}.\end{split}\]
From this observation and \eqref{e-gue201226yydh}, we get 
\begin{equation}\label{e-gue201226yyde}
\mbox{$\Bigr((\frac{\pr}{\pr z_j})^{\alpha}(\frac{\pr}{\pr \ol w_j})^{\beta}a_0\Bigr)(0,0)=0$, for every $\alpha, \beta\in\mathbb N_0$, $\alpha+\beta=N_0$}. 
\end{equation}
We can repeat the proof of \eqref{e-gue201226yyde} with minor change and deduce that 
\begin{equation}\label{e-gue201226ycd}
\mbox{$\Bigr((\frac{\pr}{\pr z})^{\alpha}(\frac{\pr}{\pr \ol w})^{\beta}a_0\Bigr)(0,0)=0$, for every $\alpha, \beta\in\mathbb N^n_0$, $|\alpha|+|\beta|=N_0$}. 
\end{equation}

Since $a_0$ is independent of $x_{2n+1}$ and $a_0(x',x)=\frac{1}{2\pi^{n+1}}$, we have 
\begin{equation}\label{e-gue201226ycdI}
\Bigr((\frac{\pr}{\pr y_{2n+1}})^Na_0\Bigr)(x,x)=0,\ \ \mbox{for every $N\in\mathbb N$}. 
\end{equation}
From \eqref{e-gue201226ycdI}, we can repeat the proof of \eqref{e-gue201226yyde} with minor change and deduce that 
\begin{equation}\label{e-gue201226ycdII}
\mbox{$\Bigr((\frac{\pr}{\pr z})^{\alpha}\frac{\pr}{\pr \ol w})^{\beta}(\frac{\pr}{\pr y_{2n+1}})^{\gamma}a_0\Bigr)(0,0)=0$, for every $\alpha, \beta, \gamma\in\mathbb N^n_0$, $|\alpha|+|\beta|+|\gamma|=N_0$}. 
\end{equation}

From \eqref{e-gue201225ycdy}, \eqref{e-gue201225ycdz}, \eqref{e-gue201226ycd}, \eqref{e-gue201226ycdII}, we get 
\[|a_0(x,y)-\frac{1}{2\pi^{n+1}}|=O(|(x,y)|^{N_0+1}).\]
By induction, we get the claim \eqref{e-gue201226yydb}. From \eqref{e-gue201226yydb}, the theorem follows. 
\end{proof}

 \subsection{Calculation of the coefficients}\label{s-gue201226yyd}

To calculate $a_1(x,x)$ in \eqref{Boutet-Sjostrand formula}, we need to choose some good coordinates. 
We will work with the assumptions as in Theorem~\ref{Boutet-Sjostrand-Hsiao formula}. Fix $p\in X$. Since $\ddbar_b$ has closed range, from~\cite{Boutet1975CRIF}, we see that there is an open set $U$ of $p$ such that $U$ can be CR embedded into $\mathbb C^{n+1}$. From now on, we identify $U$ with $\pr M\bigcap\mathring{D}$, where 
\begin{equation}\label{e-gue201226ycdIII}
\begin{split}
 &\pr M:=\{z\in\mathbb C^{n+1};\, r(z)=0\},\\
 &r(z)\in\C^\infty(\mathbb C^{n+1}, \mathbb R),\\
 &\mbox{$J(dr)=1$ on $\pr M$, $J$ is the standard complex structure on $\mathbb C^{n+1}$},\\
&\mbox{$\mathring{D}$ is an open set of $p$ in $\mathbb C^{n+1}$}. 
\end{split}
\end{equation} 
From~\cite[Lemma 3.2]{Hormander2004NO}, we can find local holomorphic coordinates $x=(x_1,\ldots,x_{2n+2})=z=(z_1,\ldots,z_{n+1})$, $z_j=x_{2j-1}+ix_{2j}$, $j=1,\ldots,n+1$, defined on $\mathring{D}$ (we assume that $\mathring{D}$ is small enough) such that 
\begin{equation}\label{e-gue201226ycdIV}
\begin{split}
&z(p)=0,\\
&r(z)=2\mathrm{Im\,}z_{n+1}+\sum_{j=1}^n|z_j|^2+O(|(z_1,\cdots,z_{n+1}|)^4).
\end{split}
\end{equation} 
From now on, we assume that \eqref{e-gue201226ycdIV} hold. It is well-known that (see~\cite[Proposition 1.1, Theorem 1.5]{BouSj76}) we can take $\phi\in\C^\infty(U\times U)$ in \eqref{Boutet-Sjostrand formula} so that $\phi$ satisfies \eqref{e-gue201217yyd} and 
\begin{equation}\label{e-gue201226ycda}
\begin{split}
&\phi=\mathring{\phi}|_{U\times U},\\
&\mathring{\phi}(z,w)=\frac{1}{i}\sum_{\alpha,\beta\in\mathbb{N}^{n+1}_0,|\alpha|+|\beta|\leq N}\frac{\partial^{\alpha+\beta}r}{\partial z^\alpha\partial\overline{z}^\beta}(0)\frac{z^\alpha}{\alpha!}\frac{\overline{w}^\beta}{\beta!}+O(|(z,w)|^{N+1}),\  \ \mbox{for every $N\in\mathbb N$}.
\end{split}
\end{equation}
From \eqref{e-gue201226ycda}, we can check that 
\begin{equation}\label{e-gue201226ycdb}
\begin{split}
&\mbox{$\ddbar_{b,x}(\phi(x,y))$ vanishes to infinite order at $x=y$},\\
&\mbox{$\ddbar_{b,y}(-\ol \phi(y,x))$ vanishes to infinite order at $x=y$}.
\end{split}
\end{equation}
From now on, we assume that $\phi$ satisfies \eqref{e-gue201217yyd} and \eqref{e-gue201226ycda}.

From implicit function theorem, if $\mathring{D}$ is small enough, we can find $R(x_1,\ldots,x_{2n+1})\in\C^\infty(D)$ such that 
\begin{equation}\label{e-gue201226ycdV}
\mbox{for every $x\in\mathring{D}$,$(x_1,\ldots,x_{2n+1})\in D$, $x\in U$ if and only if  $x_{2n+2}=R(x_1,\ldots,x_{2n+1})$},
\end{equation}
where $D$ is an open set of $\mathbb R^{2n+1}$, $0\in D$.  From now on, we assume that $\mathring{D}$ is small enough so that \eqref{e-gue201226ycdV} holds. Let $x=(x_1,\ldots,x_{2n+1})$ be local coordinates of $D$ given by the map 
\begin{equation}\label{e-gue201226ycdVI}
(x_1,\ldots,x_{2n+1})\in D\To (x_1,\ldots,x_{2n+1},R(x_1,\ldots,x_{2n+1}))\in U. 
\end{equation}
From now on, we identify $U$ with $D$ and we will work with local coordinates $x=(x_1,\ldots,x_{2n+1})$ as \eqref{e-gue201226ycdVI}. The following follows from some straightforward calculation.  We omit the details. 

\begin{proposition}
\label{geometric datum on X}
With the notations used above, we have 
\begin{equation}
\label{derivative of R at 0}
\begin{split}
&R(x)=-\frac{1}{2}\sum^{2n}_{j=1}x^2_j+O(|x|^4),\\
&\frac{\pr^2R}{\pr z_j\pr\ol z_k}(0)=-\frac{1}{2}\delta_{j,k},\ \ j,k=1,\ldots,n,
\end{split}
\end{equation}
\begin{equation}
\label{omega_0 in Chern-Morser-Hormander coordinates}
 \omega_0(x)=dx_{2n+1}-i\sum_{j=1}^n\left(\frac{\partial R}{\partial z_j}dz_j-\frac{\partial R}{\partial\overline{z}_j}d\overline{z}_j\right)+O(|x|^4),
\end{equation}
\begin{equation}
\label{d omega_0 in Chern-Morser-Hormander coordinates}
    d\omega_0(x)=2i\sum_{j=1,k}^n\frac{\partial^2 R}{\partial z_j\partial\overline{z}_k}dz_j\wedge d\overline{z}_k+O(|x|^3),
\end{equation}
\begin{equation}
\label{T^1,0 in Chern-Morser-Hormander coordinates}
    T^{1,0}_x X=\mathrm{span}\left\{\frac{\partial}{\partial z_j}+i\frac{\partial R}{\partial z_j}\frac{\partial}{\partial x_{2n+1}}+O(|x|^4)\right\}_{j=1}^n,
\end{equation}
\begin{equation}
\label{T in Chern-Morser-Hormander coordinates}
    T(x)=-\frac{\partial}{\partial x_{2n+1}}+O(|x|^2).
\end{equation}
In particular, the volume form on $X$ given by
\begin{equation}
    \lambda(x)dx_1\cdots dx_{2n+1}=\frac{1}{n!}\left(\frac{-d\omega_0}{2}\right)^n\wedge\omega_0
\end{equation}
satisfies 
\begin{equation}
\label{lambda(0)=1}
\begin{split}
&\lambda(0)=1,\\
&\frac{\pr\lambda}{\pr x_j}(0)=0,\ \ j=1,\ldots,2n+1.
\end{split}
\end{equation}
\end{proposition}
 
 We also need 
 
 \begin{proposition}
\label{phase function mathring phi}
With the same notations above, we have 
\begin{equation}\label{e-gue201227yydp}
    \phi(x,y)=-x_{2n+1}+y_{2n+1}+\frac{i}{2}\sum_{j=1}^n\left[|z_j-w_j|^2+(\overline{z}_jw_j-z_j\overline{w}_j)\right]+O\left(|(x,y)|^4\right), 
\end{equation}
\begin{equation}\label{e-gue201227yydq}
\begin{split}
&\frac{\pr^4\phi}{\pr z_j\pr z_k\pr\ol z_\ell\pr\ol z_s}(0,0)=-i\frac{\pr^4R}{\pr z_j\pr z_k\pr\ol z_\ell\pr\ol z_s}(0),\ \ j, k, \ell, s\in\{1,\ldots,n\},\\
&\frac{\pr^4\phi}{\pr w_j\pr w_k\pr\ol w_\ell\pr\ol w_s}(0,0)=-i\frac{\pr^4R}{\pr z_j\pr z_k\pr\ol z_\ell\pr\ol z_s}(0),\ \ j, k, \ell, s\in\{1,\ldots,n\},
\end{split}
\end{equation}
where $\frac{\pr}{\pr w_j}=\frac{1}{2}(\frac{\pr}{\pr y_{2j-1}}-i\frac{\pr}{\pr y_{2j}})$, $j=1,\ldots,n$, and 
\begin{equation}\label{e-gue201227yydr}
    T^2\phi(0,0)=0.
\end{equation}
\end{proposition}

\begin{proof}
From \eqref{e-gue201226ycda}, we have 
\begin{equation}
\label{Taylor formula of phi(x,y)}
    \begin{split}
    \phi(x,y)
    &=-x_{2n+1}+y_{2n+1}-i[R(x)+R(y)+\sum_{j=1}^n z_j\overline{w}_j]+O\left(|(x,y)|^4\right)\\
    &=-x_{2n+1}+y_{2n+1}+\frac{i}{2}\sum_{j=1}^n \left(|z_j|^2-2z_j\overline{w_j}+|w_j|^2\right)+O\left(|(x,y)|^4\right)\\
    &=-x_{2n+1}+y_{2n+1}+\frac{i}{2}\sum_{j=1}^n\left[|z_j-w_j|^2+(\overline{z}_jw_j-z_j\overline{w}_j)\right]+O\left(|(x,y)|^4\right)
\end{split}
\end{equation} 
and 
\begin{equation}
\label{e-gue201227yydx}
    \begin{split}
    &\phi(x,0)\\
    &=-x_{2n+1}-iR(x)\\
    &+\frac{1}{i}\sum_{\alpha_j\in\mathbb N_0, j=1,\ldots,n+1,\alpha_1+\cdots+\alpha_{n+1}=4}\frac{1}{\alpha_1!\cdots\alpha_{n+1}!}\frac{\pr^4r}{\pr z^{\alpha_1}_1\cdots\pr z^{\alpha_{n+1}}_{n+1}}(0)z^{\alpha_1}_1\cdots z^{\alpha_n}_n(x_{2n+1}+iR(x))^{\alpha_{n+1}}\\
    &\quad+O\left(|x|^5\right).
\end{split}
\end{equation}
From \eqref{e-gue201227yydx}, we get 
\begin{equation}\label{e-gue201227yydy}
\frac{\pr^4\phi}{\pr z_j\pr z_k\pr\ol z_\ell\pr\ol z_s}(0,0)=-i\frac{\pr^4R}{\pr z_j\pr z_k\pr\ol z_\ell\pr\ol z_s}(0),\ \ j, k, \ell, s\in\{1,\ldots,n\}.
\end{equation}
Similarly, 
\begin{equation}\label{e-gue201227yydz}
\frac{\pr^4\phi}{\pr w_j\pr w_k\pr\ol w_\ell\pr\ol w_s}(0,0)=-i\frac{\pr^4R}{\pr z_j\pr z_k\pr\ol z_\ell\pr\ol z_s}(0),\ \ j, k, \ell, s\in\{1,\ldots,n\}.
\end{equation}
From \eqref{Taylor formula of phi(x,y)}, \eqref{e-gue201227yydy} and \eqref{e-gue201227yydz}, we get \eqref{e-gue201227yydp} and \eqref{e-gue201227yydq}. 
   
Finally, because for all $x$ near $p$,
\begin{equation*}
    T(x)=-\frac{\partial}{\partial x_{2n+1}}+O(|x|^2),
\end{equation*}
it is clear that
\begin{equation}
    T^2\phi(0,0)=0.
\end{equation}
\end{proof}

 By Malgrange preparation theorem~\cite[Theorem 7.5.5]{Hormander2003ALPDO1} again, we have 
 \begin{equation}\label{e-gue201228yyd}
 \begin{split}
 &\phi(x,y)=f(x,y)\Phi(x,y)\ \ \mbox{on $D$},\\
&\Phi(x,y)=-x_{2n+1}+g(x',y),
\end{split}
\end{equation}
where $f(x,y), g(x',y)\in\C^\infty(D\times D)$, $x'=(x_1,\ldots,x_{2n})$. From Proposition~\ref{phase function mathring phi}, it is straightforward to check that 
\begin{equation}\label{e-gue201228yyda}
\begin{split}
&f(x,y)=1+O(|(x,y)|^3),\\
&\mbox{$\Phi(x,y)$ satisfies \eqref{e-gue201227yydp}, \eqref{e-gue201227yydq} and \eqref{e-gue201227yydr}}.
\end{split}
\end{equation} 
We write 
\begin{equation}\label{e-gue201228yydb}
\Pi(x,y)\equiv\int_0^\infty e^{it\phi(x,y)}a(x,y,t)dt\equiv\int e^{it\Phi(x,y)}A(x,y,t)dt,
\end{equation}
where $a(x,y,t), A(x,y,t)\in S^n_{\mathrm{cl}}(D\times D\times\mathbb{R}_+)$,
 \begin{equation}\label{e-gue201228yydc}
 \begin{split}
 &a(x,y,t)\sim\sum_{j=0}^\infty a_j(x,y)t^{n-j}~\text{in}~S^n_{\mathrm{cl}}(D\times D\times\mathbb{R}_+),~a_j(x,y)\in \mathscr{C}^\infty(D\times D)~\text{for all}~j\in\mathbb{N}_0,\\
    &A(x,y,t)\sim\sum_{j=0}^\infty A_j(x,y)t^{n-j}~\text{in}~S^n_{\mathrm{cl}}(D\times D\times\mathbb{R}_+),~A_j(x,y)\in \mathscr{C}^\infty(D\times D)~\text{for all}~j\in\mathbb{N}_0.
    \end{split}
 \end{equation}
We can repeat the proof of Lemma~\ref{a_0 is independent of x_2n+1} and conclude that we can take $a_j(x,y)$, $A_j(x,y)$ are independent of $x_{2n+1}$, for every $j$. From now on, we assume that $a_j(x,y)$, $A_j(x,y)$ are independent of $x_{2n+1}$, for every $j$, and hence 
\begin{equation}\label{e-gue201228yydd}
 \begin{split}
 &a_j(x,y)=a_j(x',y),\ \ j=0,1,\ldots,\\
 &A_j(x,y)=A_j(x',y),\ \ j=0,1,\ldots. 
 \end{split}
 \end{equation}
Since $\phi$ satisfies \eqref{e-gue201226ycdb}, we can repeat the proof of Theorem~\ref{a_0 is a constant} with minor change and deduce that 
\begin{equation}\label{e-gue201228yyde}
\mbox{$a_0(x',y)=\frac{1}{2\pi^{n+1}}+O(|(x,y)|^N)$, for every $N\in\mathbb N$}. 
\end{equation} 
Moreover, from Lemma~\ref{lemma of osc} and the proof of  Lemma~\ref{a_0 is independent of x_2n+1}, we can take $A_0(x',y)$ to be 
\begin{equation}\label{e-gue201228yydf}
A_0(x',y)=a_0(x',y)\frac{1}{(\Td f((x',g(x',y)),y))^{n+1}}, 
\end{equation} 
where $\Td f$ is an almost analytic extension of $f$. From \eqref{e-gue201228yyda}, \eqref{e-gue201228yyde} and \eqref{e-gue201228yydf} and , it is easy to see that 
\begin{equation}\label{e-gue201228yydg}
\mbox{$A_0(x',y)=\frac{1}{2\pi^{n+1}}+O(|(x,y)|^3)$}. 
\end{equation} 
From Lemma~\ref{l-gue201224yyd}, we see that to prove Theorem~\ref{main theorem}, we only need to calculate $a_1(0,0)$. Note $(T^2\phi)(0,0)=0$, $(T^2\Phi)(0,0)=0$. 
From this observation, we can repeat the proof of Lemma~\ref{l-gue201224yyd} and deduce that 
\begin{equation}\label{e-gue201228yydI}
a_1(0,0)=A_1(0,0). 
\end{equation} 
Hence,  to prove Theorem~\ref{main theorem}, we only need to calculate $A_1(0,0)$. Now, we are going to calculate $A_1(0,0)$. We apply the projection relation 
\begin{equation*}
    \Pi=\Pi^2,
\end{equation*}
or equivalently, in the sense of oscillatory integral
\begin{equation}
\Pi(x,y)\equiv\int_D\Pi(x,w)\Pi(w,y)\lambda(w)dw~\mathrm{mod}~\mathscr{C}^\infty(D\times D),
\end{equation}
to compute $A_1(0,0)=a_1(0,0)$, where $\lambda(w)dw$ is the volume form on $X$. Now, shrink $D$ if necessary. From
\[
\Pi(x,y)\in \mathscr{C}^\infty(X\times X\setminus\mathrm{diag}{X\times X}),
\]
 we may assume that all the base variables $x,y,w\in D$ are within a compact set. Then, in the sense of oscillatory integral,
\    \begin{equation}
\label{microlocal Pi=Pi^2}
\begin{split}
    ~
    &\int_0^\infty e^{it\Phi(x,y)}A(x,y,t)dt\\
    &\equiv\iiint_{D\times\mathbb{R}_+\times\mathbb{R}_+} e^{it\Phi(x,w)+is\Phi(w,y)}A(x,w,t)A(w,y,s)\lambda(w)dwdsdt\\
    &\equiv\int_0^\infty\left(\iint_{D\times\mathbb{R}_+} e^{it\Phi(x,w)+it\sigma\Phi(w,y)}{tA(x,w,t)A(w,y,t\sigma)}\lambda(w)dwd\sigma\right)dt\\
    &\equiv\int_0^\infty\left(\iint_{D\times\mathbb{R}_+} e^{it(-x_{2n+1}+g(x',w)+\sigma(-w_{2n+1}+g(w',y))}{tA(x,w,t)A(w,y,t\sigma)}\lambda(w)dwd\sigma\right)dt.
\end{split}
\end{equation}
    Consider the phase function 
\[
\Psi(w,\sigma,x,y):=-x_{2n+1}+g(x',w)+\sigma(-w_{2n+1}+g(w',y)).\]
    It is clear that 
$\mathrm{Im}\Psi\geq 0$ for $\sigma\geq 0$. Also, 
\[
d_w\Psi=d_wg(x',w)+\sigma(-dw_{2n+1}+d_{x'}g(w',y)),~\frac{\partial\psi}{\partial\sigma}=-w_{2n+1}+g(w',y),
\]
and with respect to $(w,\sigma)$, $\mathrm{Hess}(\Psi)$ is a matrix of the form
\[
\mathrm{Hess}(\Psi)=
\begin{bmatrix} 
\frac{\partial^2 }{\partial w^2}\left(g(x',w)+\sigma g(w',y)\right) & \bigr(\frac{\partial }{\partial w}(-w_{2n+1}+g(w',y))\bigr)^t\\
\frac{\partial }{\partial w}(-w_{2n+1}+g(w',y))& 0
\end{bmatrix}.
\]
Directly, at $(w,\sigma;x,y)=(0,1,0,0)\in\mathbb{R}^{n+1}\times\mathbb{R}_+\times\mathbb{R}^{2n+1}\times\mathbb{R}^{2n+1}$, 
$$
\Psi=0~\text{and}~d_{w,\sigma}\Psi=0
$$
and notice that $\Phi$ satisfies \eqref{e-gue201227yydp}, \eqref{e-gue201227yydq} and \eqref{e-gue201227yydr} (see \eqref{e-gue201228yyda}), we can compute that 
at $(w,\sigma;x,y)=(0,1,0,0)\in\mathbb{R}^{n+1}\times\mathbb{R}_+\times\mathbb{R}^{2n+1}\times\mathbb{R}^{2n+1}$, 
\begin{equation}\label{e-gue201228ycdIII}
    \det\mathrm{Hess}(\Psi)
      =\det
    \begin{bmatrix}
    2iI_{2n} & 0 & 0\\
        0    & 0 & -1\\
        0    & -1 & 0
    \end{bmatrix}\\
    =-\det\begin{bmatrix}
    2iI_{2n} & 0 & 0\\
        0    & 1 & 0\\
        0    & 0 & 1
    \end{bmatrix}\\
    =(-1)^{n+1}2^{2n},
\end{equation}
Hence, we can find a solution $\tilde{W}(x',y)$ and $\tilde{\Sigma}(x',y)$ near $0\in\mathbb{R}^{2n+1}$ and $1\in\mathbb{R}_+$ such that
\begin{equation}
    \label{critical value of Psi}
    \frac{\partial\tilde{\Psi}}{\partial\tilde{w}}(\tilde{W},\tilde{\Sigma},x,y)=\frac{\partial\tilde{\Psi}}{\partial\tilde{\sigma}}(\tilde{W},\tilde{\Sigma},x,y)=0.
\end{equation}
Note that $\tilde{W}(x,y)=\tilde{W}(x',y)$ and $\tilde{W}(x,y)=\tilde{\Sigma}(x',y)$ are independent of $x_{2n+1}$. So by the stationary phase theorem of Melin--Sj\"ostrand, we get
\begin{equation}
\label{integrand of Pi^2}
\begin{split}
&\int_0^\infty e^{it\Phi(x,y)}A(x,y,t)dt\\
&\equiv\iint_{D\times\mathbb{R}_+}
e^{it\Psi(w,\sigma;x,y)}{tA(x,w,t)A(w,y,t\sigma)}\lambda(w)dwd\sigma dt\\
&\equiv 
\int^{+\infty}_0e^{it\left(-x_{2n+1}+\tilde{g}\left(x',\tilde{W}(x',y)\right)\right)}{B}(x,y,t)dt,
\end{split}
\end{equation}
where 
\begin{equation}\label{e-gue2-1228yydII}
    {B}(x,y,t)\sim\sum_{j=0}^\infty {B}_j(x,y)t^{n-j}~\text{in}~S^n_{\mathrm{cl}}(D\times D\times\mathbb{R}_+),~{B}_j(x,y)\in \mathscr{C}^\infty(D\times D),\ \ \mbox{for every $j\in\mathbb N_0$}. 
\end{equation}
Since $\tilde{W}(x',y)$,  $\tilde{\Sigma}(x',y)$ and $A_j(x',y)$ are independent of $x_{2n+1}$, for every $j\in\mathbb{N}_0$, it is straightforward to see that up to $O(|x-y|^N)$, for every $N\in\mathbb N_0$, $B_j(x,y)$ to be independent of $x_{2n+1}$, for every $j\in\mathbb N_0$. Hence, 
        \begin{equation}\label{e-gue2-1228yydIII}
    {B}_j(x,y)={B}_j(x',y)+O(|x-y|^N),\ \ \mbox{for every $N\in\mathbb N$, $j\in\mathbb N_0$}. 
\end{equation}
Also, observe that
\begin{equation}\label{e-gue201228yydV}
\begin{split}
       B_0(0,0)
       &=\det\left(\frac{\mathrm{Hess}(\psi)}{2\pi i}\right)^{-\frac{1}{2}}A_0(0,0)^2\lambda(0)\\
       &=2\pi^{n+1}\left(\frac{1}{2\pi^{n+1}}\right)^2 \\
       &=\frac{1}{2\pi^{n+1}}\\
       &=A_0(0,0).
\end{split}
\end{equation}
For now $\Pi=\Pi^2$, we have
\begin{equation*}
    \int_0^\infty e^{it\left(-x_{2n+1}+g(x',y)\right)}A(x,y,t)dt\equiv\int_0^\infty e^{it\left(-x_{2n+1}+\tilde{g}\left(x',\tilde{W}(x',y)\right)\right)}B(x,y,t)dt.
\end{equation*}
Because of
\begin{equation*}
    A_0(x,y)B_0(x,y)\neq 0,
\end{equation*}
as in \cite[Section 8]{HM16}, we can show that
\begin{equation*}
\tilde{g}\left(x',\tilde{W}(x',y)\right)=g(x',y)+O(|x-y|^N),\ \ \mbox{for every $N\in\mathbb N_0$}. 
\end{equation*}
We may  replace $\tilde{g}\left(x',\tilde{W}(x',y)\right)$ by $g(x',y)$ and we have
\begin{equation}\label{e-gue201228yydVI}
    \int_0^\infty e^{it\Phi(x,y)}A(x,y,t)dt\equiv\int_0^\infty e^{it\Phi(x,y)}B(x,y,t)dt~\mathrm{mod}~\mathscr{C}^\infty(D\times D).
\end{equation}

 \begin{lemma}\label{l-gue201228yyd}
 With the notations used above, we have 
  \begin{equation}\label{e-gue2-1228yydIV}
    B_j(x,y)=A_j(x,y)+O(|x-y|^N),\ \ \mbox{for every $N\in\mathbb N$, $j\in\mathbb N_0$}. 
\end{equation}
 \end{lemma} 
 
 \begin{proof}
 We can repeat the procedure as in the discussion before \eqref{e-gue201228yydVII} and conclude that 
 \begin{equation}\label{e-gue201228yydVIII}
 B_0(x,y)-A_0(x,y)=h(x,y)(-x_{2n+1}+g(x',y))+O(|x-y|^N),\ \ \mbox{for every $N\in\mathbb N$, $j\in\mathbb N_0$},
 \end{equation}
 where $h(x,y)\in\C^\infty(D\times D)$. 
 Take almost analytic extension and $\tilde{x}_{2n+1}=g(x',y)$ in \eqref{e-gue201228yydVIII}, and notice that up to $O(|x-y|^N)$, for every $N\in\mathbb N_0$, $B_0(x,y)-A_0(x,y)$ is independent of $x_{2n+1}$, we conclude that 
 \begin{equation}\label{e-gue201228ycd}
 B_0(x,y)-A_0(x,y)=O(|x-y|^N),\ \ \mbox{for every $N\in\mathbb N$, $j\in\mathbb N_0$}.
 \end{equation}
 
 From \eqref{e-gue201228ycd},  we can repeat the procedure as in the discussion before \eqref{e-gue201228yydVII} again and conclude that 
 \begin{equation}\label{e-gue201228ycdI}
 B_1(x,y)-A_1(x,y)=h_1(x,y)(-x_{2n+1}+g(x',y))+O(|x-y|^N),\ \ \mbox{for every $N\in\mathbb N$, $j\in\mathbb N_0$},
 \end{equation}
 where $h_1(x,y)\in\C^\infty(D\times D)$. 
 Take almost analytic extension and $\tilde{x}_{2n+1}=g(x',y)$ in \eqref{e-gue201228ycdI}, and notice that up to $O(|x-y|^N)$, for every $N\in\mathbb N_0$, $B_1(x,y)-A_1(x,y)$ is independent of $x_{2n+1}$, we conclude that 
 \begin{equation*}
 B_1(x,y)-A_1(x,y)=O(|x-y|^N),\ \ \mbox{for every $N\in\mathbb N$, $j\in\mathbb N_0$}.
 \end{equation*}
Continuing in this way, the lemma follows. 
 \end{proof}

\subsection{Recursive formula between the first and the second coefficient}
From \eqref{integrand of Pi^2}, we see that 
\begin{equation}\label{e-gue201228ycdII}
\begin{split}
&t\iint_{D\times\mathbb{R}_+} e^{it(\Phi(0,w)+\sigma\Phi(w,0))}{A(0,w,t)A(w,0,t\sigma)}\lambda(w)dwd\sigma\\
&\sim B_0(0,0)t^n+B_1(0,0)t^{n-1}+\cdots.
\end{split}\end{equation}
In this section, we will from the the asymptotic expansion \eqref{e-gue201228ycdII} and Lemma~\ref{l-gue201228yyd} to get recursive formula between the first and the second coefficient.  

Now, let 
\[
F(w,\sigma):=\Phi(0,w)+\sigma\Phi(w,0)=g(0',w)+\sigma(-w_{2n+1}+g(w',0)).
\]
As \eqref{e-gue201228ycdIII} and the discussion before \eqref{e-gue201228ycdIII}, we 
have $(d_wF)(0,1)=0$, $(d_\sigma F)(0,1)=0$ and 
\begin{equation}
  \det\mathrm{Hess}(F)(0,1)
=\det\begin{bmatrix}
2iI_{2n} & 0 & 0\\
0 & 0  & -1\\
0 & -1 & 0
\end{bmatrix}=(-1)^{n+1}2^{n}  
\end{equation}
and
\begin{equation}
\mathrm{Hess}(F)^{-1}(0,1)
=\begin{bmatrix}
\frac{1}{2i}I_{2n} & 0 & 0\\
0 & 0  & -1\\
0 & -1 & 0
\end{bmatrix}
\end{equation}
and
\begin{equation}
    \left\langle \mathrm{Hess}(F)(0,1)^{-1}D,D\right\rangle=2i\sum_{j=1}^n\frac{\partial^2}{\partial z_j\partial\overline{z}_j}+2\frac{\partial^2}{\partial x_{2n+1}\partial\sigma},
\end{equation}
where $D:=\left(-i\partial_{x_1},\cdots,-i\partial_{x_{2n+1}},-i\partial_\sigma\right)^t$. 
By H\"ormander stationary phase formula Theorem \ref{Hormander stationary phase formula},
\begin{align}\label{e-gue201228ycda}
    {B}_0(0,0)t^n+{B}_1(0,0)t^{n-1}+\cdots
    &\sim t\iint_{D\times\mathbb{R}_+} e^{itF(w,\sigma)}{A(0,w,t)A(w,0,t\sigma)}\lambda(w)dwd\sigma\\
    &\sim e^{itF(0,1)}\det\left(\frac{t\mathrm{Hess}(F)(0,1)}{2\pi i}\right)^{-\frac{1}{2}}\sum_{j=0}^\infty t^{-j}P_j\\
    &\sim 2\pi^{n+1}(t^nP_0+t^{n-1}P_1+\cdots),
\end{align}
where
\begin{equation}
        P_0={A_0(0,0)^2\lambda(0)}
\end{equation}
and 
\begin{equation}
    \begin{split}
        P_1
    &=\sum_{0\leq\mu\leq 2}\frac{i^{-1}}{\mu!(\mu+1)!}\left(i\sum_{j=1}^n\frac{\partial^2}{\partial z_j\partial\overline{z}_j}+\frac{\partial^2}{\partial x_{2n+1}\partial\sigma}\right)^{\mu+1}\left(G^\mu(x,\sigma){A_0(0,x)A_0(x,0)\lambda(x)\sigma^n}\right)(0,1)\\
    &+2{A_0(0,0)A_1(0,0)\lambda(0)},
    \end{split}
\end{equation}
and
\begin{align*}
    G(x,\sigma)
    &:=F(x,\sigma)-F(0,1)-\frac{1}{2}\left\langle \mathrm{Hess}(F)(0,1)\begin{pmatrix}
x\\ \sigma-1
\end{pmatrix},\begin{pmatrix}
x\\ \sigma-1
\end{pmatrix}\right\rangle\\
&=F(x,\sigma)-\frac{1}{2}\left\langle \mathrm{Hess}(F)(0,1)\begin{pmatrix}
x\\ \sigma-1
\end{pmatrix},\begin{pmatrix}
x\\ \sigma-1
\end{pmatrix}\right\rangle
\end{align*}
satisfying 
\begin{equation}
\label{G vanishes to second order}
    \frac{\partial^\alpha G}{\partial x^{\alpha_1}\partial\sigma^{\alpha_2}}(0,1)=0\ \ \mbox{for all $\alpha_1\in\mathbb N^{2n+1}_0$, $\alpha_2\in\mathbb N_0$, $|\alpha|=|\alpha_1|+|\alpha_2|\leq 2$}.
\end{equation}
Also, from Proposition~\ref{phase function mathring phi} and \eqref{e-gue201228yyda}, we can find that
\begin{equation}
\label{G vanishes to third order in complex coordinates}
    \frac{\partial^\alpha G}{\partial x^\alpha}(0,1)=\left.\frac{\partial^\alpha}{\partial x^\alpha}\left(\Phi(0,x)+\Phi(x,0)\right)\middle|\right._{x=0}=0,\ \ \mbox{for all $\alpha\in\mathbb N^{2n+1}_0$, $|\alpha|=3$}.
\end{equation}
Also, observe that
\begin{equation}
\label{G vanishes from second order in sigma}
    \frac{\partial^\alpha G}{\partial\sigma^\alpha}(0,1)=0,\ \ \mbox{for all $\alpha\in\mathbb N_0$, $|\alpha|\geq2$}.
\end{equation}
We now calculate each terms in $P_1$:
For $\mu=0$ in the summation, the summand is   
    $$
    \frac{1}{i}\left(i\sum_{j=1}^n\frac{\partial^2}{\partial z_j\partial\overline{z}_j}+\frac{\partial^2}{\partial x_{2n+1}\partial\sigma}\right)\left({A_0(0,x)A_0(x,0)\lambda(x)\sigma^n}\right)(0,1);
    $$
for $\mu=1$ in the summation, the summand is    
    $$
    \frac{1}{2i}\left(-\sum_{j,k=1}^n\frac{\partial^4}{\partial z_j\partial\overline{z}_j\partial z_k\partial\overline{z}_k}+2i\sum_{j=1}^n\frac{\partial^4}{\partial z_j\partial\overline{z}_j\partial x_{2n+1}\partial\sigma}+\frac{\partial^4}{\partial x_{2n+1}^2\partial\sigma^2}\right)
    $$
    acting on
    $$G(x,\sigma){A_0(0,x)A_0(x,0)\lambda(x)\sigma^n}
    $$
    valuing at $(x,\sigma)=(0,1)$; and for $\mu=2$, the summand is  
    $$
    \frac{1}{12i}\left(-i\sum_{j,k,l=1}^n\frac{\partial^6}{\partial z_j\partial\overline{z}_j\partial z_k\partial\overline{z}_k\partial z_l\partial\overline{z}_l}-3\sum_{j,k=1}^n\frac{\partial^6}{\partial z_j\partial\overline{z}_j\partial z_k\partial\overline{z}_k\partial x_{2n+1}\partial\sigma}+3i\sum_{j=1}^n\frac{\partial^6}{\partial z_j\partial\overline{z}_j\partial x_{2n+1}^2\partial\sigma^2}+\frac{\partial^6}{\partial x_{2n+1}^3\partial\sigma^3}\right)
    $$
    acting on
    $$
G^2(x,\sigma){A_0(0,x)A_0(x,0)\lambda(x)\sigma^n}
    $$
    valuing at $(x,\sigma)=(0,1)$.  Thus, by  Proposition~\ref{geometric datum on X}, Proposition \ref{phase function mathring phi}, \eqref{e-gue201228yyda}, (\ref{G vanishes to second order}), (\ref{G vanishes to third order in complex coordinates}) and (\ref{G vanishes from second order in sigma}), also with \eqref{e-gue201228yydg} and Lemma~\ref{l-gue201228yyd}, it is straightforward to check that 
\begin{equation}
\label{P_1}
    \begin{split}
  &P_1=A_0(0,0)^2\left[\left(\sum_{j=1}^n\frac{\partial^2\lambda}{\partial z_j\partial\overline{z}_j}\right)(0)+\frac{i}{2}\lambda(0)\sum_{j,k=1}^n\frac{\partial^4(g(z,0)+g(0,z))}{\partial z_j\partial\overline{z}_j\partial z_k\partial\overline{z}_k}(0',0)\right]\\
    &+2{A_0(0,0)A_1(0,0)\lambda(0)}.
    \end{split}
\end{equation}
From \eqref{lambda(0)=1}, \eqref{e-gue201227yydq}, \eqref{e-gue201228yyda} and notice that $A_0(0,0)=\frac{1}{2\pi^{n+1}}$, we can rewrite \eqref{P_1}: 
\begin{equation}
\label{P_1=A_0(0,0)^2...}
\begin{split}
        P_1
        &=\frac{1}{(2\pi^{n+1})^2}\left[\sum_{j=1}^n\left(\frac{\partial^2\lambda}{\partial z_j\partial\overline{z}_j}\right)(0)+\sum_{j,k=1}^n\frac{\partial^4 R}{\partial z_j\partial\overline{z}_j\partial z_k\partial\overline{z}_k}(0)\right]+\frac{1}{\pi^{n+1}}A_1(0,0), 
\end{split}
\end{equation}
where $R$ is as in \eqref{e-gue201227yydq}. From Lemma~\ref{l-gue201228yyd} and \eqref{e-gue201228ycda}, we get 
\[A_1(0,0)=B_1(0,0)=2\pi^{n+1}P_1.\]
From this observation and \eqref{P_1=A_0(0,0)^2...}, we get 
\begin{equation}\label{e-gue201228ycdb}
    \begin{split}
    A_1(0,0)
    &={B}_1(0,0)\\
    &=2\pi^{n+1}P_1\\
    &=\frac{1}{2\pi^{n+1}}\left[\sum_{j=1}^n\left(\frac{\partial^2\lambda}{\partial z_j\partial\overline{z}_j}\right)(0)+\sum_{j,k=1}^n\frac{\partial^4 R}{\partial z_j\partial\overline{z}_j\partial z_k\partial\overline{z}_k}(0)\right]+2A_1(0,0).
    \end{split}
\end{equation}
So we need to calculate
\begin{equation}
    \label{local expression of a_1(0,0)}
    A_1(0,0)=-\frac{1}{2\pi^{n+1}}\left[\sum_{j=1}^n\left(\frac{\partial^2\lambda}{\partial z_j\partial\overline{z}_j}\right)(0)+\sum_{j,k=1}^n\frac{\partial^4 R}{\partial z_j\partial\overline{z}_j\partial z_k\partial\overline{z}_k}(0)\right]
\end{equation}
in the final subsection. 

\subsection{Calculation by geometric invariance}
In this section, we calculate each term in (\ref{local expression of a_1(0,0)}) by the geometric invariance on $X$. We will continue work with local coordinates $x=(x_1,\ldots,x_{2n+1})$ as \eqref{e-gue201226ycdVI}. We first calculate the Tanaka--Webster scalar curvature in terms of the coordinates $x=(x_1,\ldots,x_{2n+1})$. Recall that from Proposition \ref{geometric datum on X}
\begin{equation}
 \omega_0(x)=dx_{2n+1}-i\sum_{j=1}^n\left(\frac{\partial R}{\partial z_j}dz_j-\frac{\partial R}{\partial\overline{z}_j}d\overline{z}_j\right)+O(|x|^4),
\end{equation}
\begin{equation}
    d\omega_0(x)=2i\sum_{j,k=1}^n\frac{\partial^2 R}{\partial z_j\partial\overline{z}_k}dz_j\wedge d\overline{z}_k+O(|x|^3),
\end{equation}
\begin{equation}
    T(x)=-\frac{\partial}{\partial x_{2n+1}}+O(|x|^2),
\end{equation}
\begin{equation}
\begin{split}
    &T^{1,0}_x X=\mathrm{span}\left\{L_j\right\}_{j=1}^n:=\mathrm{span}\left\{\frac{\partial}{\partial z_j}+i\frac{\partial R}{\partial z_j}\frac{\partial}{\partial x_{2n+1}}+O(|x|^4)\right\}_{j=1}^n,\\
    &L_j=\frac{\partial}{\partial z_j}+i\frac{\partial R}{\partial z_j}\frac{\partial}{\partial x_{2n+1}}+O(|x|^4),\ \ j=1,\ldots,n. 
    \end{split}
\end{equation}
Write $\nabla_{L_i}L_j=\Gamma^l_{ij}L_l$, where $\nabla$ denotes the Tanaka-Webster connection (see Proposition~\ref{p-gue201228yyds}).  Recall that from \cite{Tanaka1975SPCDG}*{Lemma 3.2},
\begin{equation}
\label{Cartan magic formula}
    d\omega_0\left(\nabla_{L_i}L_j,\overline{L}_k\right)=L_i(d\omega_0(L_j,\overline{L}_k))-d\omega_0\left(L_j,[L_i,\overline{L}_k]_{T^{0,1}}\right).
\end{equation}
Directly,
\begin{equation}\label{e-gue201228ycdc}
    \begin{split}
     d\omega_0\left(\nabla_{L_i}L_j,\overline{L}_k\right)
    &=d\omega_0\left(\Gamma^l_{ij}L_l,\overline{L}_k\right)=2i\Gamma_{ij}^l\frac{\partial^2 R}{\partial z_l\partial\overline{z}_k}+O\left(|x|^3\right),
\end{split}
\end{equation}
\begin{equation}\label{e-gue201228ycdd}
    L_i\left(d\omega_0\left(L_j,\overline{L}_k\right)\right)=2\left(i\frac{\partial^3 R}{\partial z_i\partial z_j\partial\overline{z}_k}-\frac{\partial R}{\partial z_i}\frac{\partial^3 R}{\partial x_{2n+1}\partial z_j\partial\overline{z}_k}\right)+O\left(|x|^3\right),
\end{equation}
\[
\begin{split}
    [L_i,\overline{L}_k]
    &=\left[\frac{\partial}{\partial z_i}+i\frac{\partial R}{\partial z_i}\frac{\partial}{\partial x_{2n+1}}+O(|x|^4),\frac{\partial}{\partial \overline{z}_k}-i\frac{\partial R}{\partial\overline{z}_k}\frac{\partial}{\partial x_{2n+1}}+O(|x|^4)\right]\\
    &=\left(\frac{\partial R}{\partial z_i}\frac{\partial^2 R}{\partial\overline{z}_k\partial x_{2n+1}}-\frac{\partial R}{\partial\overline{z}_k}\frac{\partial^2 R}{\partial z_i\partial x_{2n+1}}-2i\frac{\partial^2 R}{\partial z_i\partial\overline{z}_k}\right)\frac{\partial}{\partial x_{2n+1}}+O(|x|^3),
\end{split}
\]
and hence
\begin{equation}\label{e-gue201228ycde}
   d\omega_0\left(L_j,[L_i,\overline{L}_k]_{T^{0,1}}\right)=O(|x|^3). 
\end{equation}
Accordingly, by (\ref{derivative of R at 0}), for all $i,j,k=1,\cdots,n$,
\begin{equation}
    \label{Christoffel symbol}
    \Gamma^k_{ij}(0)=0.
\end{equation}
Moreover, by taking $\frac{\partial}{\partial\overline{z}_h}$ both sides in (\ref{Cartan magic formula}), from (\ref{derivative of R at 0}), \eqref{e-gue201228ycdc}, \eqref{e-gue201228ycdd} and \eqref{e-gue201228ycde}, it is not difficult to check that 
\begin{equation}
    \label{derivative of Christoffel symbol}
    \frac{\partial\Gamma^k_{ij}}{\partial\overline{z}_h}(0)=-2\frac{\partial^4 R}{\partial z_i\partial z_j\partial\overline{z}_k\partial\overline{z}_h}(0).
\end{equation}
Now, let $\{\theta^\alpha\}_{j=1}^n$ and $\{\theta^{\overline{\beta}}\}_{j=1}^n$ be the dual frame of $\{L_\alpha\}_{j=1}^n$ and $\{\overline{L}_\beta\}_{j=1}^n$, respectively.
Denote
\[
\nabla L_\alpha=\omega_\alpha^\beta\otimes L_\beta,
\]
and we can check that
the $(1,1)$ part of $d\omega_\alpha^\beta$ is
\[
-\sum_{k,l=1}^n \left(\overline{L}_l\Gamma^\beta_{k\alpha})\right)\theta^k\wedge\theta^{\overline{l}}+O(|x|),
\]
and the $(1,1)$ part of $\Theta_\alpha^\beta=d\omega_\alpha^\beta-\omega_\alpha^\gamma\wedge\omega_\gamma^\beta$ denoted by
\[
\sum_{k,l=1}^n R_{\alpha k\overline{l}}^\beta\theta^k\wedge\theta^{\overline{l}}
\]
equals the $(1,1)$ part ot $d\omega_\alpha^\beta$. Hence the pseudohermitian Ricci curvature tensor at origin is 
\[
R_{\alpha\overline{l}}(0)=\sum_{k=\beta=1}^n R_{\alpha k\overline{l}}^\beta(0)=-\sum_{k=\beta=1}^n\frac{\partial\Gamma^\beta_{k\alpha}}{\partial\overline{z}_l}(0)=2\sum_{k=1}^n\frac{\partial^4 R}{\partial z_k\partial\overline{z}_k\partial z_\alpha\partial\overline{z}_l}(0).
\]
Also, for
\[
-d\omega_0=ig_{\alpha\overline{\beta}}\theta^\alpha\wedge\theta^{\overline{\beta}},
\]
we can find that $\theta^\alpha(0)=dz_\alpha$ and $\theta^{\overline{\beta}}(0)=d\overline{z}_\beta$, and 
\[
g_{\alpha\overline{\beta}}(0)=\delta_{\alpha\beta}.
\]
Let $g^{\overline{c}d}$ be the inverse matrix of $g_{a\overline{b}}$. We have 
 $g^{\overline{c}d}(0)=\delta_{cd}$ and the Tanaka--Webster scalar curvature at the origin is
\begin{equation}
\label{Tanaka--Webster Scalar curvature at origin}
        R_{\mathrm{scal}}(0)
    =g^{\overline{l}\alpha}R_{\alpha\overline{l}}(0)
    =2\sum_{l=1}^n\sum_{k=1}^n\frac{\partial^4 R}{\partial z_l\partial\overline{z}_l\partial z_k\partial\overline{z}_k}(0).
\end{equation}
Finally, for the volume form 
$$
\lambda(x)dx:=\frac{1}{n!}\left(\left(\frac{-d\omega_0}{2}\right)^n\wedge\omega_0\right),
$$
we have the expression
\begin{equation}
\begin{split}
~
        &\frac{1}{n!}\left(\left(-\sum_{j,k=1}^n i\frac{\partial^2 R}{\partial z_j\partial\overline{z}_k}dz_j\wedge d\overline{z}_k+O\left(|x|^3\right)\right)^n\wedge\left(dx_{2n+1}-i\sum_{j=1}^n\left(\frac{\partial R}{\partial z_j}dz_j-\frac{\partial R}{\partial\overline{z}_j}d\overline{z}_j\right)+O\left(|x|^4\right)\right)\right)\\
        &=\frac{1}{n!}\left(\left(\sum_{j,k=1}^n -2\frac{\partial^2 R}{\partial z_j\partial\overline{z}_k}\frac{dz_j\wedge d\overline{z}_k}{-2i}+O\left(|x|^3\right)\right)^n\wedge\left(dx_{2n+1}-i\sum_{j=1}^n\left(\frac{\partial R}{\partial z_j}dz_j-\frac{\partial R}{\partial\overline{z}_j}d\overline{z}_j\right)+O\left(|x|^4\right)\right)\right).
\end{split}
\end{equation}
From Proposition \ref{geometric datum on X}, we can check that 
\begin{equation}
\label{complex laplacian on volume at origin}
    \frac{\partial^2\lambda}{\partial z_l\partial\overline{z}_l}(0)=(-2)^n\left(-\frac{1}{2}\right)^{n-1}\sum_{j=1}^n\frac{\partial^4 R}{\partial z_l\partial\overline{z}_l\partial z_j\partial\overline{z}_j}(0)=-2\sum_{j=1}^n\frac{\partial^4 R}{\partial z_l\partial\overline{z}_l\partial z_j\partial\overline{z}_j}(0).
\end{equation}
From \eqref{local expression of a_1(0,0)}, \eqref{Tanaka--Webster Scalar curvature at origin} and \eqref{complex laplacian on volume at origin}, 
we conclude that
\begin{equation}
\label{A_1(0,0)}
    A_1(0,0)=\frac{1}{4\pi^{n+1}}{R_{\mathrm{scal}}(0)}.
\end{equation}
Thus, the proof for (2) in Theorem \ref{main theorem} is completed for the point-wise equation (\ref{A_1(0,0)}) holds for all $x_0\in D$ and $R_{\mathrm{scal}}$ is globally well-defined on whole $X$, and in particular on $D$.
\section*{Acknowledgement}
The authors would like to thank George Marinescu for  several discussions about this work. 

\bibliographystyle{plain}

\end{document}